\documentclass[11pt,reqno]{amsart}
\usepackage{amsmath,mathtools}
\usepackage{mathrsfs}
\usepackage{graphicx}
\usepackage{epsfig,epsf,psfrag}
\usepackage{amsfonts}
\usepackage{setspace}
\usepackage{color}
\usepackage[font=small]{caption}
\usepackage[font=footnotesize]{subcaption}
\usepackage[top=1in, bottom=1in, left=1.2in, right=1.2in]{geometry}
\usepackage{afterpage}
\usepackage{bm}
\usepackage{multicol}
\usepackage{multirow}
\usepackage{diagbox}
\usepackage{booktabs}
\usepackage{algorithm}
\usepackage{algorithmic}
\usepackage{lineno}
\usepackage{indentfirst}
\usepackage{cases}

\usepackage{isomath}

\newtheorem{theorem}{Theorem}[section]
\newtheorem{definition}[theorem]{Definition}
\newtheorem{corollary}[theorem]{Corollary}
\newtheorem{lemma}[theorem]{Lemma}

\newtheorem{remark}[theorem]{Remark}

\newcommand{\bbeta}{\boldsymbol{\beta}}

\newcommand{\brho}{\boldsymbol{\rho}}

\newcommand{\btau}{\boldsymbol{\tau}}

\newcommand{\bphi}{\boldsymbol{\phi}}

\newcommand{\bpsi}{\boldsymbol{\psi}}
\newcommand{\bomega}{\boldsymbol{\omega}}

\newcommand{\bPhi}{\boldsymbol{\Phi}}

\allowdisplaybreaks

\title[]{Shape Taylor expansion for wave scattering problems} 

\numberwithin{equation}{section}

\author{Gang Bao}
\address{School of Mathematical Sciences, Zhejiang University,
	Hangzhou, Zhejiang 310027, China}
\email{baog@zju.edu.cn}

\author{Haoran Ma}
\address{School of Mathematical Sciences, Zhejiang University,
	Hangzhou, Zhejiang 310027, China}
\email{MaHaoran77@zju.edu.cn}

\author{Jun Lai}
\address{School of Mathematical Sciences, Zhejiang University,
	Hangzhou, Zhejiang 310027, China}
\email{laijun6@zju.edu.cn}

\author{Jingzhi Li}
\address{Department of Mathematics, Southern University of Science and Technology,
	Shenzhen, Guangdong 518055, China}
\email{li.jz@sustech.edu.cn}

\subjclass[2020]{35Q61, 35J05, 35R30, 49J50, 78M50}

\keywords{shape derivatives, domain derivatives, shape optimizations, wave scattering, inverse wave scattering}

\begin{document}

\begin{abstract}
 The Taylor expansion of wave fields with respect to shape parameters has a wide range of applications in wave scattering problems, including inverse scattering, optimal design, and uncertainty quantification. However, deriving the high order shape derivatives required for this expansion poses significant challenges with conventional methods. This paper addresses these difficulties by introducing elegant recurrence formulas for computing high order shape derivatives. The derivation employs tools from exterior differential forms, Lie derivatives, and material derivatives. The work establishes a unified framework for computing the high order shape perturbations in scattering problems. In particular, the recurrence formulas are applicable to both acoustic and electromagnetic scattering models under a variety of boundary conditions, including Dirichlet, Neumann, impedance, and transmission types. 
\end{abstract}

\maketitle

\section{Introduction}\label{introduction}

As a fundamental tool to explore the effect of shape perturbations, the shape derivative plays a crucial role in wave scattering. Its applications span a wide range of areas, including shape optimizations~\cite{bao2014optimal,ma2024inverse}, inverse scattering~\cite{hettlich1995frechet,hettlich1996iterative}, uncertainty quantification~\cite{dolz2020higher,hao2018computation,harbrecht2013first}, and more. One standard approach to solving inverse obstacle problems is the iterative method based on first order shape gradient~\cite{hagemann2019solving}. However, there are many applications requiring higher order perturbations for more accurate approximations, especially in wave scattering from scatterers with small scale structures and inverse scattering problems requiring faster convergence rates~\cite{bui2012analysis1,bui2013analysis3,hagemann2020application}. Despite this need, research on the theory of second or higher order shape derivatives is extremely limited due to the complicated derivation process by standard approaches. To bridge this gap, this paper focuses on developing the high order shape differential theory for scattering problems.

Analogous to classical differential calculus, 
the concept of differentiability of wave fields with respect to (\textit{w.r.t.}) the boundaries of scatterers has a long history~\cite{colton2019inverse}. Early work on shape derivatives in wave scattering mainly focused on acoustics. For instance, \cite{kirsch1993domain} employed the domain derivative method to solve the inverse obstacle scattering problem with a sound soft boundary condition, and \cite{hettlich1995frechet} investigated the existence and characterization of Fr\'{e}chet derivatives for transmission and Robin boundary problems. Subsequently, shape derivatives for electromagnetic fields were developed in works such as~\cite{costabel2012shape,haddar2004frechet,potthast1996domain}. In particular, \cite{potthast1996domain} proved that the solution to the scattering problem is infinitely differentiable w.r.t. the boundary of the obstacle based on the integral equation approach, and ~\cite{haddar2004frechet} established Fr\'{e}chet differentiability of the solution under the impedance boundary condition. Currently, a large number of related works~\cite{bao2015inverse,bao2019inverse,feijoo2003application,hagemann2019solving,kress2018inverse} show that shape derivatives are crucial tools for applications requiring sensitivity analysis. However, methods for deriving shape derivatives remain relatively limited. In particular, the material derivative method introduced by~\cite{kirsch1993domain} continues to be a primary tool, but it is often tailored to specific equations and boundary conditions. Even for the same equation, methods based on variational formulations~\cite{hettlich2012domain}, and boundary integral equations~\cite{costabel2012shape} may have very different (albeit equivalent) forms, resulting in unnecessary confusion.

In contrast to the extensive attention given to first order shape derivatives, significantly less focus has been placed on second order or higher order derivatives~\cite{afraites2008second,bui2012analysis1}. Nevertheless, high order expansions are essential in many applications~\cite{bui2013analysis3,dolz2020higher,hettlich1999second,hagemann2020application}. For example,~\cite{hettlich1999second} demonstrated that the inverse scattering algorithm based on second order shape derivatives achieves greater efficiency compared to first order methods. Similarly, \cite{dolz2020higher} showed that higher order expansion techniques exhibit improved stability in random scattering problems. However, computing the higher order shape Taylor expansion of scattered fields remains a significant challenge. A primary reason is the increased complexity of deriving higher order shape perturbations, which is considerably more tedious than first order derivations.  The intricacy can be found, for instance, in~\cite{dolz2020higher,hettlich1999second}. This computational complexity has hindered the broader development and adoption of higher order perturbation methods. Thus, in this paper, we aim to provide a systematic approach for high order shape derivatives in acoustic and electromagnetic scattering problems.

Our work is motivated by~\cite{hiptmair2013shape}, which introduced an elegant geometric approach using differential forms and Lie derivatives to study the shape derivatives of boundary value problems. Based on these geometric tools,~\cite{hiptmair2018shape} systematically extended the derivation and established results for first order shape derivatives in acoustic and electromagnetic scattering models under various commonly used boundary conditions. The contribution of our work lies in generalizing this approach to higher order shape derivatives and proposing formulas for the shape Taylor expansion in scattering problems.The significance of Taylor expansions for the scattered field with respect to the boundaries of scatterers is analogous to that of Taylor expansions in classical differential theory: they fully reveal the connection between the local properties of a shape functional and its shape derivatives. The primary challenge of this task is the derivation complexity raised from surface differential operators, including surface gradient, divergence, and curl. While surface gradient operators effectively simplify the expression of first order derivatives~\cite{hiptmair2018shape}, they are less suitable for deriving higher order cases. To address this difficulty, we employ the trace operator decomposition, which eliminates the need for surface differential operators. This approach facilitates the recursive derivation of higher order shape derivatives, significantly simplifying the process.

In the context of shape optimization, it is important to distinguish between shape sensitivity and shape derivatives, as they represent related but distinct concepts~\cite{adriaens2021adjoint}. Shape sensitivity refers to the gradient of an objective function, which depends on the scattered field, with respect to shape parameters.  Its primary computational tool is the adjoint-state method~\cite{adriaens2021adjoint,ma2024inverse}, which relies on shape derivatives during the derivation for shape optimizations. Another related concept is the topological derivative, which studies the perturbation effect of topological changes on the scattered field. This concept has found applications in shape reconstruction~\cite{feijoo2004new} and topology optimization~\cite{sisamon2012acoustic}. Shape derivatives serve as an important theoretical basis for topological derivatives~\cite{novotny2019applications}. However, this paper does not address topological derivative theory, as its focus is exclusively on the derivation of high order shape derivatives for scattered fields.

This paper is organized as follows: in Section~\ref{sect2}, we introduce the definitions and essential geometric tools related to shape derivatives and shape Taylor expansion. Section~\ref{sect3} provides a detailed derivation of recursive shape derivative formulas using exterior differential forms. This section consists of two subsections, where subsection~\ref{DiriBound} focuses on wave equations with Dirichlet boundary conditions, also referred to as sound soft boundary in acoustics and perfect electric conductor (PEC) in electromagnetics. Subsection~\ref{NeumannBound} addresses the shape derivative for Neumann boundary conditions, known as sound hard boundary in acoustics and perfect magnetic conductor (PMC) in electromagnetics. 
Extension of the analysis to impedance and transmission boundary conditions is given in Section~\ref{sec4}. Section~\ref{VecPro} presents formulas for the second order shape derivatives under the vector proxies in Euclidean space for acoustics and electromagnetic equations. The paper is concluded in Section~\ref{secCon}.

\section{Notations and definitions}\label{sect2}
In this section, we recall some necessary notations and tools related to differential forms and other essential concepts introduced by~\cite{hiptmair2018shape,hiptmair2013shape} to describe the main conclusions of shape derivatives in scattering problems. Detailed definitions and operating rules for these forms can also be found in~\cite{sokolowski1992introduction}.   
\begin{figure}
    \centering
    \includegraphics[width = 0.25\textwidth]{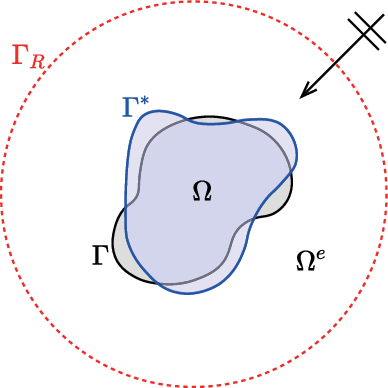}
    \caption{The scattering model with a perturbed boundary.}
    \label{FIGmodel}
\end{figure}

Suppose that $\Omega$ is a bounded domain in $\mathbb{R}^d$ with a smooth boundary $\Gamma:=\partial \Omega$, where $d\ge1$ is the dimension. Let $l$ be an integer satisfying $0\le l\le d$. The set containing all $l-$forms is abbreviated as $\land^l$. The scattered field for acoustics and electromagnetics can be represented as a differential $l-$form $\bomega$ defined in  $\Omega^e:=\mathbb{R}^d\backslash \bar{\Omega}$, with the corresponding equation given by
\begin{eqnarray}\label{EqD}
    (-1)^{d-l}\mathbf{d}(\star_\alpha\mathbf{d}\bomega) + (-1)^d\star_{k^2}\bomega = \mathbf{0}\qquad&{\rm in}\quad \Omega^e,
\end{eqnarray}
where $\star_\alpha$ and $\star_{k^2}$ are Hodge-$\star$ operators. The subscript $\alpha$ and $k^2$ of $\star$ indicate different Riemannian metrics, with $\alpha$ representing the medium parameter such as material density, permittivity, permeability, etc., and $k$ representing the wavenumber. The bold symbol $\mathbf{d}$ denotes the exterior differential operator. Detailed definitions for the Hodge-$\star$ operator, $\mathbf{d}$, and exterior differential forms can be found in~\cite{frankel2011geometry, hiptmair2013shape}.
 
 The boundary condition on $\Gamma$, denoted by $BC(\cdot)$, depends on the types of scatterers. It is assumed to be given by
 \begin{eqnarray}\label{EqD2}
     BC(\bomega) = \bphi\qquad&{\rm on} \quad \Gamma,
 \end{eqnarray}
   where the right-hand side, $\bphi$, is determined by the incident field.
   
  In general, the scattered field is defined in an unbounded region and must satisfy radiation conditions as $|\mathbf{x}|\rightarrow\infty$. Here we employ a transparent boundary condition given on a truncated spherical boundary with radius $R$, denoted by $\Gamma_R$ with interior $B_R$, to simplify the far field condition of the field: 
  \begin{eqnarray}\label{EqD3}
      TBC(\bomega) = \bpsi\quad&{\rm on} \quad \Gamma_R.
  \end{eqnarray}
  Without causing ambiguity, we slightly abuse the notation by denoting $B_R\backslash \bar{\Omega}$ as $\Omega^e$. A typical scattering model is illustrated in Fig.~\ref{FIGmodel}. 

\begin{remark}
When $d=2$ and $l=0$, equation~\eqref{EqD} is equivalent to the planar acoustic field equation or the vertically polarized electric (magnetic) field equation in $\mathbb{R}^2$. When $d=3$ and $l=0$,~\eqref{EqD} corresponds to the spatial acoustic field equation in $\mathbb{R}^3$. When $d=3$ and $l=1$,~\eqref{EqD}  represents the electric (magnetic) field equation in $\mathbb{R}^3$. 
\end{remark}

We employ the velocity field method presented in~\cite{sokolowski1992introduction} to study the perturbation effect of $\Gamma$. Specifically, given a time parameter $t\in[0,T]$ and a regular domain $D\subset\mathbb{R}^d$, we introduce a velocity field  $\mathbf{v}(t,\mathbf{x})\in\mathbf{C}^{\infty}([0,T]\times\mathbb{R}^d,\mathbb{R}^d)$. Let $T^{\mathbf{v}}_t: \mathbb{R}^d\rightarrow\mathbb{R}^d$ denote the diffeomorphism induced by $\mathbf{v}$ and $t$, which is given by
    \begin{eqnarray}
    T^{\mathbf{v}}_t(\mathbf{X}) = \mathbf{x}(t,\mathbf{X})\mbox{ with }\frac{d\mathbf{x}(t,\mathbf{X})}{dt} = \mathbf{v}(t,\mathbf{x})\mbox{ and } \mathbf{x}(0,\mathbf{X}) = \mathbf{X}\in \mathbb{R}^d.
    \end{eqnarray}
Let $D_t: = T^{\mathbf{v}}_t(D)$. It yields that ${T_t^{\mathbf{v}}}^\star:\land^l(D_t)\rightarrow\land^l(D)$ is a pull-back map induced by $T^{\mathbf{v}}_t$. For each velocity field $\mathbf{v}$, denote the contraction operator induced by $\mathbf{v}$ as $\mathbf{i}_\mathbf{v}:\land^l(D)\rightarrow\land^{l-1}(D)$.

Now let us recall the definition and computational formula of Lie derivative~\cite{frankel2011geometry}:
    \begin{definition}
        For any $l-$form $\bomega$ defined on $D$, the Lie derivative of $\bomega$ w.r.t. a velocity field $\mathbf{v}$ is defined as
        \begin{eqnarray}
            \mathcal{L}_{\mathbf{v}}\bomega: = \lim_{t\rightarrow0}\frac{{T_t^{\mathbf{v}}}^\star\bomega - \bomega}{t},
        \end{eqnarray}
        which can be calculated by Cartan's formula~\cite{delfour2011shapes}:
        \begin{eqnarray}\label{Cartan}
            \mathcal{L}_\mathbf{v}\bomega = (\mathbf{i}_\mathbf{v}\mathbf{d} + \mathbf{d}\mathbf{i}_\mathbf{v})\bomega.
        \end{eqnarray}
    \end{definition}
The contraction operator $\mathbf{i}_\mathbf{v}$ and the exterior differential operator $\mathbf{d}$ satisfy the distributive law under the operation of exterior product $\land$:
\begin{subequations}\label{distrlaw}
\begin{align}
    \mathbf{i}_\mathbf{v}(\bomega_1\land\bomega_2) =& \mathbf{i}_\mathbf{v}\bomega_1\land\bomega_2 + (-1)^{l_1}\bomega_1\land\mathbf{i}_\mathbf{v}\bomega_2,\\
    \mathbf{d}(\bomega_1\land\bomega_2) =& \mathbf{d}\bomega_1\land\bomega_2 + (-1)^{l_1}\bomega_1\land\mathbf{d}\bomega_2,
\end{align}
\end{subequations}
where $\bomega_1\in\land^{l_1}$, $\bomega_2\in\land^{l_2}$ with $l_1, l_2\in \mathbb{N}$. To simplify the discussion, all velocity fields are assumed to be time-independent, in which case we can rewrite $\mathbf{v}(t,\mathbf{x})$ as $\mathbf{v}(\mathbf{x})$. When $D = \Omega^e$ in equation~\eqref{EqD} and $T^{\mathbf{v}}_t$ transforms $D$ to $D_t$, $\bomega(D)$ is also transformed to $\bomega(D_t)$. We are interested in finding the shape derivative of $\bomega(D_t)$ at $t=0$. The precise definitions of first and second order shape derivatives are given by: 
    \begin{definition}
        Let $\mathbf{v}$ be a time-independent velocity field in $\mathbb{R}^d$. For an $l-$form $\bomega$ depending on $D_t: = T^{\mathbf{v}}_t(D)$, the shape derivative of $\bomega$ w.r.t. $\mathbf{v}$ is defined as
        \begin{eqnarray}
            \delta_{\mathbf{v}}\bomega = \lim_{t\rightarrow0}\frac{\bomega(D_t) - \bomega(D)}{t}.
        \end{eqnarray}
    \end{definition}
    \begin{definition}
        Let $\mathbf{v}_1$ and $\mathbf{v}_2$ be two time-independent velocity fields in $\mathbb{R}^d$, and $t_1$ and $t_2$ be two independent time parameters. For an $l-$form $\bomega$ depending on $D_{[t_1,t_2]}:=T^{\mathbf{v}_2}_{t_2}\circ T^{\mathbf{v}_1}_{t_1}(D)$, the second order shape derivative of $\bomega$ w.r.t. $\mathbf{v}_1$ and $\mathbf{v}_2$ is defined as
        \begin{eqnarray}
            \delta_{[\mathbf{v}_1,\mathbf{v}_2]}\bomega = 2\lim_{t_1,t_2\rightarrow 0}\frac{\bomega\left(D_{[t_1,t_2]}\right) - \bomega(D) - t_1\delta_{\mathbf{v}_1}\bomega  - t_2\delta_{\mathbf{v}_2}\bomega}{t_1 t_2}.
        \end{eqnarray}
    \end{definition}
    Following the definitions above, one can inductively define the higher order shape derivatives $\delta_{[\mathbf{v}_1,\dots,\mathbf{v}_N]}\bomega$ for $N$ time-independent velocity fields $\mathbf{v}_1,\dots,\mathbf{v}_N$. In general, we assume the velocity field $\mathbf{v}$ has a compact support in $\mathbb{R}^d$, with its support located within a strip-like region along the boundary $\Gamma$. Since $\mathbf{v}$ depends on $\Gamma$, we also need to define the shape derivative of a velocity field w.r.t. another velocity field:
    \begin{definition}
        Let $\mathbf{v}$ and $\mathbf{w}$ be two time-independent velocity fields defined in $\mathbb{R}^d$. The shape derivative of $\mathbf{v}$ w.r.t. $\mathbf{w}$ is defined as
        \begin{eqnarray}
            \delta_{\mathbf{w}}\mathbf{v} = \lim_{t\rightarrow0}\frac{\mathbf{v}(T^\mathbf{w}_t(D)) - \mathbf{v}(D)}{t}.
        \end{eqnarray}
    \end{definition}

    After introducing the definitions of shape derivatives, we are ready to give the definition of the shape Taylor expansion. For a single velocity field $\mathbf{v}$ and time parameter $t$, the scattered fields induced by $\Gamma$ and $\Gamma_t: = \partial D_t$ are denoted by $\bomega$ and $\bomega_t$, respectively. The shape Taylor expansion of $\bomega_t$ based on $\bomega$, $\mathbf{v}$, and $t$ is given by
\begin{eqnarray}\label{Taysingle}
    \bomega_t = \bomega + t\delta_{\mathbf{v}}\bomega + \frac{t^2}{2}\delta_{[\mathbf{v},\mathbf{v}]}\bomega + \dots + \frac{t^N}{N!}\delta_{[\mathbf{v},\dots,\mathbf{v}]_{1\times N}}\bomega + o(t^{N}).
\end{eqnarray}
Further, let $\{\mathbf{v}_1, \mathbf{v}_2,\dots, \mathbf{v}_m\}$ be $m$ time-independent velocity fields defined in $\mathbb{R}^d$ and $\{t_1,t_2,\dots,t_m\}$ be $m$ independent time parameters. The perturbed boundary is given by
\begin{eqnarray}\label{movingboun}
    \Gamma_{[t_1,\dots,t_m]} := T^{\mathbf{v}_m}_{t_m}\circ T^{\mathbf{v}_{m-1}}_{t_{m-1}}\circ \dots \circ T^{\mathbf{v}_1}_{t_1}(\Gamma).
\end{eqnarray}
Denote $\bomega_{[t_1,\dots,t_m]}$ as the scattered field induced by the perturbed boundary $\Gamma_{[t_1,\dots,t_m]}$.  The multivariable Taylor expansion is then given by
\begin{eqnarray}\label{Taymulti}
\begin{aligned}
    \bomega_{[t_1,\dots,t_m]} =& \bomega + \sum_{j = 1}^{m}t_j\delta_{\mathbf{v}_j}\bomega +\dots+\frac{1}{N!}\sum_{j_1,\dots,j_N = 1}^{m}\Big(\prod \limits_{k=1}^N t_{j_k}\Big)\delta_{[\mathbf{v}_{j_1},\dots,\mathbf{v}_{j_N}]}\bomega\\
    &+ o\Big(P_N(t_1,\dots,t_m)\Big),
\end{aligned}
\end{eqnarray}
where $P_N(\cdot)$ is a polynomial of order $N$ in $m$ variables .

According to equations~\eqref{Taysingle} and~\eqref{Taymulti}, it is evident that shape derivatives are the key components of the shape Taylor expansions. In the following, we will derive the shape derivatives under four types of boundary conditions: Dirichlet, Neumann, impedance, and transmission conditions. The results for Dirichlet and Neumann conditions represent the main contributions of this paper, while the results for impedance and transmission conditions can be seen as corollaries of the Dirichlet and Neumann cases. Consequently, we provide a detailed derivation for the Dirichlet and Neumann cases in Section~\ref{sect3}, and the discussion of the impedance and transmission cases is presented in Section~\ref{sec4}.

\section{Shape derivatives for Dirichlet and Neumann boundary conditions}\label{sect3}
This section provides the derivation of shape derivatives of all orders under the Dirichlet and Neumann boundary conditions. The following lemma is fundamental to the derivation. 
\begin{lemma}\label{lemma21}
    Given a velocity field $\mathbf{w}$, define $\mathcal{J}_t(\bomega):=\int_{D_t}\mathbf{i}_{\mathbf{v}(D_t)}\bomega(D_t)$ as a domain functional on the moving geometry $D_t$. The material derivative of $\mathcal{J}_t(\bomega)$ w.r.t. $\mathbf{w}$ at $t=0$ is given by
    
    \begin{eqnarray}\label{deW1}
        \lim_{t\rightarrow0}\frac{\mathcal{J}_t(\bomega) - \mathcal{J}_0(\bomega)}{t} = \int_D \mathbf{i}_{\delta_\mathbf{w} \mathbf{v}}\bomega + \mathbf{i}_\mathbf{v}\delta_\mathbf{w}\bomega + \mathcal{L}_\mathbf{w}\int_D\mathbf{i}_\mathbf{v}\bomega.
    \end{eqnarray}
\end{lemma}

\begin{proof}
    Recall the definitions of $\mathbf{i}$, ${T_t^{\mathbf{w}}}^\star$ and $\mathcal{L}_\mathbf{w}$ given in Section~\ref{sect2}. The material derivative at $t=0$ can be computed by~\cite{hiptmair2013shape}
    \begin{eqnarray}
        \begin{aligned}
        \lim_{t\rightarrow0}\frac{\mathcal{J}_t(\bomega) - \mathcal{J}_0(\bomega)}{t}=&\lim_{t\rightarrow0}\frac{1}{t}\left[\int_{D_t}\mathbf{i}_{\mathbf{v}(D_t)}\bomega(D_t) - \int_D\mathbf{i}_{\mathbf{v}(D)}\bomega(D)\right]\\
        =&\lim_{t\rightarrow0}\frac{1}{t}\left[\int_D T_t^{\mathbf{w}^\star}\mathbf{i}_{\mathbf{v}(D_t)}\bomega(D_t)-\mathbf{i}_{\mathbf{v}(D)}\bomega(D)\right]\\
        =&\lim_{t\rightarrow0}\int_D T_t^{\mathbf{w}^\star}\frac{\mathbf{i}_{\mathbf{v}(D_t)}-\mathbf{i}_{\mathbf{v}(D)}}{t}\bomega(D_t)\\
        &+\lim_{t\rightarrow0}\int_D T_t^{\mathbf{w}^\star}\mathbf{i}_{\mathbf{v}(D)}\frac{\bomega(D_t) - \bomega(D)}{t}\\
        &+\lim_{t\rightarrow0}\frac{1}{t}\left[\int_D T_t^{\mathbf{w}^\star}\mathbf{i}_{\mathbf{v}(D)}\bomega(D)-\mathbf{i}_{\mathbf{v}(D)}\bomega(D)\right]\\
        =&\int_D \mathbf{i}_{\delta_\mathbf{w}\mathbf{v}}\bomega+\mathbf{i}_\mathbf{v}\delta_\mathbf{w}\bomega+\mathcal{L}_\mathbf{w}\mathbf{i}_\mathbf{v}\bomega.
        \end{aligned}
    \end{eqnarray}
    \ 
\end{proof}

The normal vector field $\mathbf{n}$ of the surface $\Gamma_t$ can be treated as a velocity field in $\mathbf{R}^d$ by smooth extension. Consider the surface functional $\mathcal{J}_t(\bomega) = \int_{\Gamma_t}\mathbf{i}_{\mathbf{n}}\bomega(\Gamma_t)$. According to Lemma~\ref{deW1}, the material derivative of $\mathcal{J}_t(\bomega)$ w.r.t. $\mathbf{w}$ at $t = 0$ is given by
\begin{eqnarray}\label{deW2}
    \lim_{t\rightarrow0}\frac{\mathcal{J}_t(\bomega) - \mathcal{J}_0(\bomega)}{t} = \mathcal{L}_\mathbf{w}\mathcal{J}_0(\bomega) + \int_\Gamma\delta_\mathbf{w}(\mathbf{i}_\mathbf{n}\bomega),
\end{eqnarray}
where $\delta_\mathbf{w}(\mathbf{i}_\mathbf{n}\bomega) = \mathbf{i}_{\delta_\mathbf{w}\mathbf{n}}\bomega + \mathbf{i}_\mathbf{n}\delta_\mathbf{w}\bomega$. More generally, we have the following corollary. 
\begin{corollary}\label{coro22}
Suppose that $F(\bomega;\mathbf{v},\mathbf{n}):\land^l\rightarrow\mathbb{C}$ is a linear functional of $\bomega$, with the dependence of $\mathbf{v}$ and $\mathbf{n}$ for $F$ being through the contraction operator $\mathbf{i}_{\mathbf{v}}$ and $\mathbf{i}_{\mathbf{n}}$ acting on the differential form $\bomega$. If $F$ is shape differentiable w.r.t. another velocity field $\mathbf{w}$, then the shape derivative operator $\delta_{\mathbf{w}}$ applied to $F$ is given by
\begin{eqnarray}\label{deW3}
\begin{aligned}
    \delta_{\mathbf{w}}F(\bomega;\mathbf{v},\mathbf{n}) =& F(\delta_{\mathbf{w}}\bomega;\mathbf{v},\mathbf{n}) + F(\bomega;\delta_{\mathbf{w}}\mathbf{v},\mathbf{n}) + F(\bomega;\mathbf{v},\delta_{\mathbf{w}}\mathbf{n})\\
    =&\delta_{\mathbf{w}}^{\bomega}F(\bomega;\mathbf{v},\mathbf{n}) + \delta_{\mathbf{w}}^\mathbf{v}F(\bomega;\mathbf{v},\mathbf{n}) + \delta_{\mathbf{w}}^{\mathbf{n}}F(\bomega;\mathbf{v},\mathbf{n}).
\end{aligned}
\end{eqnarray}
\end{corollary}

Now we proceed to derive the recurrence formulas for the shape derivatives of scattered fields $\bomega$ up to any order $N$. In particular, for the purpose of shape perturbation analysis, the normal velocity field~\eqref{veloField2} is sufficient to represent general velocity fields, since the tangential components of a general velocity field on $\Gamma$ do not affect the shape of the boundary~\cite{sokolowski1992introduction}. To ease the discussion, we first assume all velocity fields are restricted to the form
\begin{eqnarray}\label{veloField1}
    \mathbf{v}(\mathbf{x}) = v\mathbf{n}(\mathbf{x}),\qquad \mathbf{x}\in\Gamma,
\end{eqnarray}
where $v$ is a constant. This assumption ensures that $\delta_\mathbf{v}\mathbf{n}\equiv\vec{0}$ \cite{nedelec2001acoustic}. We then extend the conclusion to the general case:
\begin{eqnarray}\label{veloField2}
    \mathbf{v}(\mathbf{x}) = v(\mathbf{x})\mathbf{n}(\mathbf{x}),\quad \mathbf{x}\in\Gamma.
\end{eqnarray}

\subsection{Dirichlet boundary}\label{DiriBound}

For a scattering problem with Dirichlet boundary conditions, also known as the sound soft boundary condition in acoustics and the perfect electric conductor (PEC) condition in electromagnetics, one can rewrite equations~\eqref{EqD}-\eqref{EqD3} as a system of first order equations
\begin{eqnarray}\label{EqS}
\left\{
\begin{aligned}
    &\mathbf{d}\brho + (-1)^l\star_{k^2}\bomega = \mathbf{0}&&{\rm in}\quad\Omega^e,\\
    &\star_{\alpha^{-1}}\brho - (-1)^{(l+1)(d-1)}\mathbf{d}\bomega=\mathbf{0}&&{\rm in}\quad\Omega^e,\\
    &\star^\Gamma\mathbf{Tr}^{\mathcal{D}}(\bomega)=\star^\Gamma\mathbf{Tr}^{\mathcal{D}}(\bphi) = \star^\Gamma\mathbf{i}_\mathbf{n}\star\bphi&&{\rm on}\quad \Gamma,\\
    &\mathcal{F}(\brho,\bomega) = \mathbf{0}&&{\rm on}\quad \Gamma_R,
\end{aligned}
\right.
\end{eqnarray}
where $\brho$ is a $(d-l-1)-$form, and the metric $\alpha^{-1}$ is the inverse metric of $\alpha$, satisfying:
\begin{eqnarray}
    \brho = \star_\alpha\mathbf{d}\bomega\quad\Leftrightarrow\quad \star_{\alpha^{-1}}\brho = (-1)^{(l+1)(d-1)}\mathbf{d}\bomega.
\end{eqnarray}
The third equation in~\eqref{EqS} represents the Dirichlet boundary condition, where we denote $\mathbf{Tr}^{\mathcal{D}}: = \mathbf{i}_\mathbf{n}\star$ the Dirichlet trace operator. It should be noted that if the metric of the Hodge-$\star$ operator is the standard metric, we simply omit the subscript on $\star$. If $\star$ is defined on the surface $\Gamma$, we denote it by $\star^{\Gamma}:\land^{l}\rightarrow\land^{d-l-1}$. The incident field is given by $\bphi$. The last equation in~\eqref{EqS} represents the transparent boundary condition on $\Gamma_R$, which is fixed during the derivation of the shape derivative.

Let us introduce the trace operator decomposition, which decompose the restriction of $\bomega$ on $\Gamma$ as
\begin{eqnarray}\label{separate}
\bomega|_\Gamma = \star^\Gamma\mathbf{Tr}^{\mathcal{D}}(\bomega) + \star\mathbf{i}_\mathbf{n}\star\mathbf{i}_\mathbf{n}\bomega.
\end{eqnarray}
This is an orthogonal decomposition of $\bomega$ into the normal and tangential components relative to $\Gamma$. Specifically, for a scalar field $u$ and a vector field $\mathbf{E}$, the corresponding decompositions are:
\begin{eqnarray}
    u|_\Gamma = \mathbf{n}\cdot\mathbf{n}u - \mathbf{n}\times\mathbf{n}u,\quad\mathbf{E}|_\Gamma = -\mathbf{n}\times\mathbf{n}\times\mathbf{E} + \mathbf{n}(\mathbf{n}\cdot\mathbf{E}).
\end{eqnarray}

The recurrence formula for the shape derivatives of $\bomega$ under the Dirichlet boundary condition is given by the following theorem.
\begin{theorem}\label{ThemD}
    Suppose that $\bomega$ is the solution of equation~\eqref{EqS}. Let $\delta_{\mathbf{v}_{[N]}}\bomega$ be the $N$th order shape derivative w.r.t $\mathbf{v}_{[N]}:=[\mathbf{v}_1,\dots,\mathbf{v}_{N}]$ for $N=0,1,\dots$, and $\delta_{\mathbf{v}_{[N]}}\bomega=\bomega$ for $N=0$. Then the $N+1$th order shape derivative $\delta_{\mathbf{v}_{[N+1]}}\bomega$ satisfies the same equation~\eqref{EqS}, except that the boundary condition on $\Gamma$ is replaced by
    \begin{eqnarray}
        \begin{aligned}
        \star^\Gamma\mathbf{Tr}^{\mathcal{D}}(\delta_{\mathbf{v}_{[N+1]}}\bomega)=&\delta_{\mathbf{v}_{N+1}}^{\bomega}\big(\star^\Gamma\mathbf{Tr}^\mathcal{D}(\delta_{\mathbf{v}_{[N]}}\bomega)\big) + \delta_{\mathbf{v}_{N+1}}^{\mathbf{n}}\big(\star^\Gamma\mathbf{Tr}^\mathcal{D}(\delta_{\mathbf{v}_{[N]}}\bomega)\big)\\
        &+\mathbf{i}_{\mathbf{v}_{N+1}}\mathbf{d}\big(\star^\Gamma\mathbf{Tr}^\mathcal{D}(\delta_{\mathbf{v}_{[N]}}\bomega)\big)\\
        &-\mathbf{i}_{\mathbf{v}_{N+1}}\mathbf{d}\star^\Gamma\mathbf{i}_\mathbf{n}\star\delta_{\mathbf{v}_{[N]}}\bomega - \star^\Gamma\mathbf{i}_{\delta_{\mathbf{v}_{N+1}}\mathbf{n}}\star\delta_{\mathbf{v}_{[N]}}\bomega,
        \end{aligned}
    \end{eqnarray}
    where  $\star^\Gamma\mathbf{Tr}^\mathcal{D}(\delta_{\mathbf{v}_{[N]}}\bomega)$ is the Dirichlet boundary condition  for $\delta_{\mathbf{v}_{[N]}}\bomega$ on $\Gamma$. 
\end{theorem}
\begin{proof}

Given a test function $\btau\in\land^{d-l-1}$, the second equation in ~\eqref{EqS} gives 
\begin{eqnarray}
\int_{\Omega^e}\star_{\alpha^{-1}}\brho\land\btau - (-1)^{(l+1)(d-1)}\mathbf{d}\bomega\land\btau = 0.\label{EqW1}
\end{eqnarray}
Let $\Gamma^-$ and $\Gamma_R^+$ be the oriented surfaces directed towards the interior of $\Omega$ and the exterior of $B_R$, respectively. By applying integration by parts and Stokes' theorem, equation~\eqref{EqW1} becomes:
\begin{eqnarray}\label{zeroorder}
\begin{aligned}
    &\int_{\Omega^e}\star_{\alpha^{-1}}\brho\land\btau - (-1)^{(l+1)d} \int_{\Omega^e}\bomega\land\mathbf{d}\btau\\
    =& (-1)^{(l+1)(d-1)}\Big[\int_{\Gamma^-}\bomega\land\btau + \int_{\Gamma_R^+}\bomega\land\btau\Big].
\end{aligned}
\end{eqnarray}
Denote $\mathcal{M}(\brho,\bomega;\btau)$ the left-hand side of equation~\eqref{zeroorder}, i.e.
\begin{eqnarray}\label{M2}
    \mathcal{M}(\brho,\bomega;\btau): = \int_{\Omega^e}\star_{\alpha^{-1}}\brho\land\btau - (-1)^{(l+1)d} \int_{\Omega^e}\bomega\land\mathbf{d}\btau.
\end{eqnarray}
By combining decomposition~\eqref{separate}, equation~\eqref{zeroorder} can be rewritten as
\begin{eqnarray}\label{weakform0}
\begin{aligned}
    \mathcal{M}(\brho,\bomega;\btau)= &(-1)^{(l+1)(d-1)}\Big[\int_{\Gamma^-}\star^\Gamma\mathbf{Tr}^{\mathcal{D}}(\bphi)\land\btau\\
    &+ \int_{\Gamma^-}\star\mathbf{i}_\mathbf{n}\star\mathbf{i}_\mathbf{n}\bomega\land\btau + \int_{\Gamma_R^+}\bomega\land\btau\Big].
\end{aligned}
\end{eqnarray}
Suppose $\Omega^e$ is a moving domain w.r.t. the velocity field $\mathbf{v}$ given by equation~\eqref{veloField1}. Since $\mathcal{M}(\brho,\bomega;\btau)$ is a domain functional of $\Omega^e$, according to Lemma~\ref{lemma21} and Cartan's formula~\eqref{Cartan}, the material derivative of $\mathcal{M}(\brho,\bomega;\btau)$ w.r.t. $\mathbf{v}$ is given by
\begin{eqnarray}\label{SG1l}
    \begin{aligned}
        \delta_\mathbf{v}\mathcal{M}(\brho,\bomega;\btau) =& \mathcal{M}(\delta_\mathbf{v}\brho,\delta_\mathbf{v}\bomega;\btau) + \mathcal{L}_\mathbf{v}\mathcal{M}(\brho,\bomega;\btau)\\
    =&\mathcal{M}(\delta_\mathbf{v}\brho,\delta_\mathbf{v}\bomega;\btau) + \int_{\Gamma^-}\mathbf{i}_\mathbf{v}\Big(\star_{\alpha^{-1}}\brho\land\btau - (-1)^{(l+1)d}\bomega\land\mathbf{d}\btau\Big).
    \end{aligned}
\end{eqnarray}
On the other hand, given the explicit form of the velocity field in equation~\eqref{veloField1}, taking the material derivative of equation~\eqref{weakform0} and making use of the facts that $\delta_{\mathbf{v}}\bphi = \mathbf{0}$ and $\mathbf{i}_{\mathbf{v}}\mathbf{d}\int_{\Gamma_R}(\cdot) = 0$, we obtain 
\begin{eqnarray}\label{SG1}
\begin{aligned}
    \delta_\mathbf{v}\mathcal{M}(\brho,\bomega;\btau)=& (-1)^{(l+1)(d-1)}\Big[\int_{\Gamma^-}\mathbf{i}_\mathbf{v}\mathbf{d}\big((\star^\Gamma\mathbf{Tr}^\mathcal{D}(\bphi) + \star\mathbf{i}_\mathbf{n}\star\mathbf{i}_\mathbf{n}\bomega)\land\btau\big)\\
    &+ \int_{\Gamma^-}\star\mathbf{i}_\mathbf{n}\star\mathbf{i}_\mathbf{n}\delta_\mathbf{v}\bomega\land\btau + \int_{\Gamma_R^+}\delta_\mathbf{v}\bomega\land\btau\Big].
\end{aligned}
\end{eqnarray}
Combining equations~\eqref{SG1l},~\eqref{SG1} and the distributive law~\eqref{distrlaw} yields
\begin{eqnarray}\label{Eq1ord}
    \begin{aligned}
    \mathcal{M}(\delta_\mathbf{v}\brho,\delta_\mathbf{v}\bomega;\btau)
    =&(-1)^{(l+1)(d-1)}\int_{\Gamma^-}\big(\mathbf{i}_\mathbf{v}\mathbf{d}\star^\Gamma\mathbf{Tr}^\mathcal{D}(\bphi) + \mathbf{i}_\mathbf{v}\mathbf{d}\star\mathbf{i}_\mathbf{n}\star\mathbf{i}_\mathbf{n}\bomega) - \mathbf{i}_\mathbf{v}\mathbf{d}\bomega\big)\land\btau\\
    & + (-1)^{(l+1)d}\int_{\Gamma^-}\big(-\mathbf{d}\bomega + \mathbf{d}\star^\Gamma\mathbf{Tr}^{\mathcal{D}}(\bphi) + \mathbf{d}\star\mathbf{i}_\mathbf{n}\star\mathbf{i}_\mathbf{n}\bomega\big)\land\mathbf{i}_\mathbf{v}\btau\\
    & + (-1)^{(l+1)d}\int_{\Gamma^-}\mathbf{i}_\mathbf{v}\big(\bomega - \star^\Gamma\mathbf{Tr}^\mathcal{D}(\bphi) - \star\mathbf{i}_\mathbf{n}\star\mathbf{i}_\mathbf{n}\bomega\big)\land\mathbf{d}\btau\\
    & + (-1)^{(l+1)(d-1)}\int_{\Gamma^-}\big(\bomega - \star^\Gamma\mathbf{Tr}^\mathcal{D}(\bphi) - \star\mathbf{i}_\mathbf{n}\star\mathbf{i}_\mathbf{n}\bomega\big)\land\mathbf{i}_\mathbf{v}\mathbf{d}\btau\\
    &+(-1)^{(l+1)(d-1)}\Big[\int_{\Gamma^-}\star\mathbf{i}_\mathbf{n}\star\mathbf{i}_\mathbf{n}\delta_\mathbf{v}\bomega\land\btau + \int_{\Gamma_R^+}\delta_\mathbf{v}\bomega\land\btau\Big]\\
    =&(-1)^{(l+1)(d-1)}\Big[\int_{\Gamma^-}\big(\mathbf{i}_\mathbf{v}\mathbf{d}\star^\Gamma\mathbf{Tr}^{\mathcal{D}}(\bphi) -\mathbf{i}_\mathbf{v}\mathbf{d}\star^\Gamma\mathbf{i}_\mathbf{n}\star\bomega\big)\land\btau\\
    &+\int_{\Gamma^-}\star\mathbf{i}_\mathbf{n}\star\mathbf{i}_\mathbf{n}\delta_\mathbf{v}\bomega\land\btau + \int_{\Gamma_R^+}\delta_\mathbf{v}\bomega\land\btau\Big],
    \end{aligned}
\end{eqnarray}
where the second step in equation~\eqref{Eq1ord} utilizes the following identities on $\Gamma$
\begin{eqnarray}
\left\{
    \begin{aligned}
        &-\mathbf{d}\bomega+ \mathbf{d}\star\mathbf{i}_\mathbf{n}\star\mathbf{i}_\mathbf{n}\bomega = -\mathbf{d}\star^\Gamma\mathbf{i}_{\mathbf{n}}\star\bomega,\\
        &\mathbf{i}_{\mathbf{n}}\star\Big(\big(\mathbf{d}\star^\Gamma\mathbf{Tr}(\bphi)-\mathbf{d}\star^\Gamma\mathbf{i}_{\mathbf{n}}\star\bomega\big)\land\mathbf{i}_\mathbf{v}\btau\Big)=\mathbf{0},\\
        &\bomega - \star^{\Gamma}\mathbf{Tr}^{\mathcal{D}}(\bphi) - \star\mathbf{i}_{\mathbf{n}}\star\mathbf{i}_{\mathbf{n}}\bomega = \mathbf{0}.
    \end{aligned}
\right.
\end{eqnarray}
According to equations~\eqref{EqS} and ~\eqref{weakform0}, the variational form~\eqref{Eq1ord} is equivalent to
\begin{eqnarray}\label{Dbd1}
\left\{
    \begin{aligned}
    &\mathbf{d}\delta_\mathbf{v}\brho + (-1)^l\star_{k^2}\delta_\mathbf{v}\bomega = \mathbf{0}&&{\rm in}\quad\Omega^e,\\
    &\star_\alpha^{-1}\delta_\mathbf{v}\brho - (-1)^{(l+1)(d-1)}\mathbf{d}\delta_\mathbf{v}\bomega=\mathbf{0}&&{\rm in}\quad\Omega^e,\\
    &\star^\Gamma\mathbf{Tr}^{\mathcal{D}}(\delta_\mathbf{v}\bomega)=\mathbf{i}_\mathbf{v}\mathbf{d}\star^\Gamma\mathbf{i}_\mathbf{n}\star\bphi -\mathbf{i}_\mathbf{v}\mathbf{d}\star^\Gamma\mathbf{i}_\mathbf{n}\star\bomega&&{\rm on}\quad \Gamma,\\
    &\mathcal{F}(\delta_\mathbf{v}\brho,\delta_\mathbf{v}\bomega) = \mathbf{0}&&{\rm on}\quad \Gamma_R,
    \end{aligned}
    \right.
\end{eqnarray}
which is the equation satisfied by the first order shape derivative $\delta_{\mathbf{v}}\bomega$. 

To derive the second order shape derivative, we further introduce another velocity field $\mathbf{w}$ and assume $\Gamma_t$ with $\Gamma_0 = \Gamma$ is a moving boundary given by~\eqref{movingboun} with $\mathbf{v}_1 = \mathbf{v}$ and $\mathbf{v}_2 = \mathbf{w}$. In order to calculate $\delta_\mathbf{w}\mathcal{M}(\delta_\mathbf{v}\brho,\delta_\mathbf{v}\bomega;\btau)$, we proceed in a manner similar to the derivation of the first order shape derivative. Based on Lemma~\ref{lemma21} and Corollary~\ref{coro22}, let us take the material derivative w.r.t. $\mathbf{w}$ on both sides of equation~\eqref{Eq1ord}:
\begin{eqnarray}\label{Eq2ord}
\begin{aligned}
    &\mathcal{M}(\delta_{[\mathbf{v},\mathbf{w}]}\brho,\delta_{[\mathbf{v},\mathbf{w}]}\bomega;\btau) + \mathcal{M}(\delta_{\delta_\mathbf{w}\mathbf{v}}\brho,\delta_{\delta_\mathbf{w}\mathbf{v}}\bomega;\btau) + \mathcal{L}_\mathbf{w}\mathcal{M}(\delta_\mathbf{v}\brho,\delta_\mathbf{v}\bomega;\btau)\\
    =&(-1)^{(l+1)(d-1)}\Big[\int_{\Gamma^-}\mathbf{i}_\mathbf{w}\mathbf{d}\big((\star^\Gamma\mathbf{Tr}^\mathcal{D}(\delta_\mathbf{v}\bomega) + \star\mathbf{i}_\mathbf{n}\star\mathbf{i}_\mathbf{n}\delta_\mathbf{v}\bomega)\land\btau\big)+\int_{\Gamma_R^+}\delta_{[\mathbf{v},\mathbf{w}]}\bomega\land\btau\\
    &+\int_{\Gamma^-}\delta_\mathbf{w}^{\bomega}(\star^\Gamma\mathbf{Tr}^\mathcal{D}(\delta_\mathbf{v}\bomega))\land\btau + \star\mathbf{i}_\mathbf{n}\star\mathbf{i}_\mathbf{n}\delta_{[\mathbf{v},\mathbf{w}]}\bomega\land\btau\\
    &+\int_{\Gamma^-}\star^{\Gamma}\mathbf{Tr}^\mathcal{D}(\delta_{\delta_\mathbf{w}\mathbf{v}}\bomega)\land\btau + \int_{\Gamma^-}\star\mathbf{i}_\mathbf{n}\star\mathbf{i}_\mathbf{n}\delta_{\delta_\mathbf{w}\mathbf{v}}\bomega\land\btau + \int_{\Gamma_R^+}\delta_{\delta_\mathbf{w}\mathbf{v}}\bomega\land\btau\Big].
\end{aligned}
\end{eqnarray}
Considering that $\star^\Gamma\mathbf{Tr}^{\mathcal{D}}(\delta_\mathbf{v}\bomega)$ satisfies the third equation in~\eqref{Dbd1}, and the incident field $\bphi$ is independent of the boundary perturbation, we obtain
\begin{eqnarray}\label{widedelta}
    \begin{aligned}
        \delta_\mathbf{w}^{\bomega}(\star^\Gamma\mathbf{Tr}^{\mathcal{D}}(\delta_\mathbf{v}\bomega)) =& \delta_\mathbf{w}^{\bomega}(\mathbf{i}_\mathbf{v}\mathbf{d}\star^\Gamma\mathbf{i}_\mathbf{n}\star\bphi -\mathbf{i}_\mathbf{v}\mathbf{d}\star^\Gamma\mathbf{i}_\mathbf{n}\star\bomega)\\
        =& -\mathbf{i}_\mathbf{v}\mathbf{d}\star^\Gamma\mathbf{i}_\mathbf{n}\star\delta_\mathbf{w}\bomega.
    \end{aligned}
\end{eqnarray}
In equation~\eqref{Eq2ord}, $\delta_\mathbf{w}\mathbf{v}$ can also be seen as a velocity field defined in $\mathbb{R}^d$. By replacing $\mathbf{v}$ with $\delta_\mathbf{w}\mathbf{v}$ in equation~\eqref{Eq1ord}, it yields
 \begin{eqnarray}\label{Eliminate}
 \begin{aligned}
     \mathcal{M}(\delta_{\delta_\mathbf{w}\mathbf{v}}\brho,\delta_{\delta_\mathbf{w}\mathbf{v}}\bomega;\btau)=&(-1)^{(l+1)(d-1)}\Big[\int_{\Gamma^-}\star^{\Gamma}\mathbf{Tr}^\mathcal{D}(\delta_{\delta_\mathbf{w}\mathbf{v}}\bomega)\land\btau\\
     &+ \int_{\Gamma^-}\star\mathbf{i}_\mathbf{n}\star\mathbf{i}_\mathbf{n}\delta_{\delta_\mathbf{w}\mathbf{v}}\bomega\land\btau + \int_{\Gamma_R^+}\delta_{\delta_\mathbf{w}\mathbf{v}}\bomega\land\btau\Big].
 \end{aligned}
 \end{eqnarray}
By plugging equation~\eqref{Eliminate} into equation~\eqref{Eq2ord}, one can eliminate the terms containing $\delta_{\mathbf{w}}\mathbf{v}$, and equation~\eqref{Eq2ord} is simplified as
\begin{eqnarray}\label{Betafunction}
\begin{aligned}
    &\mathcal{M}(\delta_{[\mathbf{v},\mathbf{w}]}\brho,\delta_{[\mathbf{v},\mathbf{w}]}\bomega;
    \btau)\\
    =&-\int_\Gamma\mathbf{i}_\mathbf{w}(\star_{\alpha^-1}\delta_\mathbf{v}\brho\land\btau - (-1)^{(l+1)d}\delta_\mathbf{v}\bomega\land d\btau)\\
    &+(-1)^{(l+1)(d-1)}\Big[\int_{\Gamma^-}\mathbf{i}_\mathbf{w}\mathbf{d}\big((\star^\Gamma\mathbf{Tr}^\mathcal{D}(\delta_\mathbf{v}\bomega) + \star\mathbf{i}_\mathbf{n}\star\mathbf{i}_\mathbf{n}\delta_\mathbf{v}\bomega)\land\btau\big)\\
    &+\int_{\Gamma^-}\delta_\mathbf{w}^{\bomega}(\star^\Gamma\mathbf{Tr}^\mathcal{D}(\delta_\mathbf{v}\bomega))\land\btau+\int_{\Gamma^-}\star\mathbf{i}_\mathbf{n}\star\mathbf{i}_\mathbf{n}\delta_{[\mathbf{v},\mathbf{w}]}\bomega\land\btau +\int_{\Gamma_R^+}\delta_{[\mathbf{v},\mathbf{w}]}\bomega\land\btau\Big].
\end{aligned}
\end{eqnarray}
Now let us make use of the following identities on $\Gamma$:
\begin{eqnarray}\label{Betadefination}
    \left\{
    \begin{aligned}
        &-\mathbf{d}\delta_{\mathbf{v}}\bomega+ \mathbf{d}\star\mathbf{i}_\mathbf{n}\star\mathbf{i}_\mathbf{n}\delta_{\mathbf{v}}\bomega = -\mathbf{d}\star^\Gamma\mathbf{i}_{\mathbf{n}}\star\delta_{\mathbf{v}}\bomega,\\
        &\mathbf{i}_{\mathbf{n}}\star\Big(\big(\mathbf{d}\star^\Gamma\mathbf{Tr}^{\mathcal{D}}(\delta_{\mathbf{v}}\bomega)-\mathbf{d}\star^\Gamma\mathbf{i}_{\mathbf{n}}\star\delta_{\mathbf{v}}\bomega\big)\land\mathbf{i}_\mathbf{w}\btau\Big)=\mathbf{0},\label{Betadefination1}\\
        &\delta_{\mathbf{v}}\bomega - \star^\Gamma\mathbf{Tr}^{\mathcal{D}}(\delta_{\mathbf{v}}\bomega) - \star\mathbf{i}_{\mathbf{n}}\star\mathbf{i}_{\mathbf{n}}\delta_{\mathbf{v}}\bomega = \mathbf{0},\label{Betadefination2}
    \end{aligned}
    \right.
\end{eqnarray}
and introduce the boundary term on $\Gamma$
\begin{eqnarray}
    \begin{aligned}
        \star^\Gamma\mathbf{Tr}^{\mathcal{D}}(\delta_{[\mathbf{v},\mathbf{w}]}\bomega)=& \mathbf{i}_\mathbf{w}\mathbf{d}\mathbf{i}_\mathbf{v}\mathbf{d}\star^\Gamma\mathbf{i}_\mathbf{n}\star\bphi - \mathbf{i}_\mathbf{w}\mathbf{d}\mathbf{i}_\mathbf{v}\mathbf{d}\star^\Gamma\mathbf{i}_\mathbf{n}\star\bomega\qquad\\
      &-\mathbf{i}_\mathbf{w}\mathbf{d}\star^\Gamma\mathbf{i}_\mathbf{n}\star\delta_\mathbf{v}\bomega -\mathbf{i}_\mathbf{v}\mathbf{d}\star^\Gamma\mathbf{i}_\mathbf{n}\star\delta_\mathbf{w}\bomega.
    \end{aligned}
\end{eqnarray}
One can then rewrite equation~\eqref{Betafunction} as
\begin{eqnarray}
\begin{aligned}
    \mathcal{M}(\delta_{[\mathbf{v},\mathbf{w}]}\brho,\delta_{[\mathbf{v},\mathbf{w}]}\bomega;
    \btau)=&(-1)^{(l+1)(d-1)}\Big[\int_{\Gamma^-}\star^\Gamma\mathbf{Tr}^{\mathcal{D}}(\delta_{[\mathbf{v},\mathbf{w}]}\bomega)\land\btau\\
    &+\int_{\Gamma^-}\star\mathbf{i}_\mathbf{n}\star\mathbf{i}_\mathbf{n}\delta_{[\mathbf{v},\mathbf{w}]}\bomega\land\btau +\int_{\Gamma_R^+}\delta_{[\mathbf{v},\mathbf{w}]}\bomega\land\btau\Big],
\end{aligned}
\end{eqnarray}
which implies that the second order shape derivative $\delta_{[\mathbf{v},\mathbf{w}]}\bomega$ satisfies the equation
\begin{eqnarray}
    \left\{
    \begin{aligned}
        &\mathbf{d}\delta_{[\mathbf{v},\mathbf{w}]}\brho + (-1)^l\star_{k^2}\delta_{[\mathbf{v},\mathbf{w}]}\bomega = \mathbf{0} &&{\rm in}\quad\Omega^e,\\
        &\star_\alpha^{-1}\delta_{[\mathbf{v},\mathbf{w}]}\brho - (-1)^{(l+1)(d-1)}\mathbf{d}\delta_{[\mathbf{v},\mathbf{w}]}\bomega=\mathbf{0}&&{\rm in}\quad\Omega^e,\\
        &\star^\Gamma\mathbf{i}_{\mathbf{n}}\star\delta_{[\mathbf{v},\mathbf{w}]}\bomega = \star^\Gamma\mathbf{Tr}^{\mathcal{D}}(\delta_{[\mathbf{v},\mathbf{w}]}\bomega)&&{\rm on}\quad\Gamma,\\
        &\mathcal{F}(\delta_{[\mathbf{v},\mathbf{w}]}\brho,\delta_{[\mathbf{v},\mathbf{w}]}\bomega)  = \mathbf{0} &&{\rm on}\quad \Gamma_R.
    \end{aligned}
    \right.
\end{eqnarray}

We now inductively derive the higher order shape derivatives w.r.t a sequence of velocity fields $\mathbf{v}_{[N+1]}: = [\mathbf{v}_1,\mathbf{v}_2,\dots,\mathbf{v}_{N+1}]$. Suppose that the boundary condition of $\delta_{\mathbf{v}_{[N]}}\bomega$ on $\Gamma$ has been given by $\star^\Gamma\mathbf{Tr}^\mathcal{D}(\delta_{\mathbf{v}_{[N]}}\bomega)$, which means
\begin{eqnarray}\label{Eqns1ord}
\begin{aligned}
        \mathcal{M}(\delta_{\mathbf{v}_{[N]}}\brho,\delta_{\mathbf{v}_{[N]}}\bomega;\btau) =& (-1)^{(l+1)(d-1)}\Big[\int_{\Gamma^-} \star^\Gamma\mathbf{Tr}^\mathcal{D}(\delta_{\mathbf{v}_{[N]}}\bomega)\land\btau\\
        &+\int_{\Gamma^-}\star\mathbf{i}_\mathbf{n}\star\mathbf{i}_{\mathbf{n}}\bomega\land\btau +\int_{\Gamma_R^+}\delta_{\mathbf{v}_{[N]}}\bomega\land\btau\Big].
\end{aligned}
\end{eqnarray}
 To derive the equation satisfied by $\delta_{\mathbf{v}_{[N+1]}}\bomega$, let us take the material derivative w.r.t. $\mathbf{v}_{N+1}$ in equation~\eqref{Eqns1ord}. Following the same procedure as deriving the second order shape derivative in equation~\eqref{Eq2ord}, we obtain
    \begin{eqnarray}\label{Eqnord}
        \begin{aligned}
            & \mathcal{M}(\delta_{\mathbf{v}_{[N+1]}}\brho,\delta_{\mathbf{v}_{[N+1]}}\bomega;\btau)\\
            =& - \int_{\Gamma^-}\mathbf{i}_{\mathbf{v}_{N+1}}\Big(\big(\star_{\alpha^{-1}}\delta_{\mathbf{v}_{[N]}}\brho - (-1)^{(l+1)(d-1)}\mathbf{d}\delta_{\mathbf{v}_{[N]}}\bomega\big)\land\btau\Big)\\
            & + (-1)^{(l+1)(d-1)}\Big[\int_{\Gamma^-}\delta_{\mathbf{v}_{N+1}}^{\bomega}\big(\star^\Gamma\mathbf{Tr}^\mathcal{D}(\delta_{\mathbf{v}_{[N]}}\bomega)\land\btau\big) + \mathbf{i}_{\mathbf{v}_{N+1}}\mathbf{d}\big(\star^\Gamma\mathbf{Tr}^\mathcal{D}(\delta_{\mathbf{v}_{[N]}}\bomega)\land\btau\big)\Big]\\
            & + (-1)^{(l+1)(d-1)}\Big[\int_{\Gamma^-}\mathbf{i}_{\mathbf{v}_{N+1}}\mathbf{d}(\star\mathbf{i}_\mathbf{n}\star\mathbf{i}_\mathbf{n}\delta_{\mathbf{v}_{[N]}}\bomega\land\btau) + \star\mathbf{i}_\mathbf{n}\star\mathbf{i}_\mathbf{n}\delta_{\mathbf{v}_{[N+1]}}\bomega\land\btau\Big]\\
            &+(-1)^{(l+1)(d-1)}\int_{\Gamma_R^+}\delta_{\mathbf{v}_{[N+1]}}\bomega\land\btau\\
            &- \sum_{j = 1}^N\mathcal{M}(\delta_{\mathbf{v}_{[j,N]}}\brho,\delta_{\mathbf{v}_{[j,N]}}\bomega;\btau) + (-1)^{(l+1)(d-1)}\sum_{j=1}^{N}\int_{\Gamma^-}\star^\Gamma\mathbf{Tr}^\mathcal{D}(\delta_{\mathbf{v}_{[j,N]}}\bomega)\land\btau\\
            &+(-1)^{(l+1)(d-1)}\sum_{j=1}^{N}\Big[\int_{\Gamma^-}\star\mathbf{i}_\mathbf{n}\star\mathbf{i}_\mathbf{n}\delta_{\mathbf{v}_{[j,N]}}\bomega\land\btau + \int_{\Gamma_R^+}\delta_{\mathbf{v}_{[j,N]}}\bomega\land\btau\Big],
        \end{aligned}
    \end{eqnarray}
    where we denote $\mathbf{v}_{[j,N]} = [\mathbf{v}_1,\dots,\mathbf{v}_{j-1},\delta_{\mathbf{v}_{N+1}}\mathbf{v}_j,\mathbf{v}_{j+1},\dots, \mathbf{v}_N]$.
    By assumption, it holds the identities 
    \begin{eqnarray}
        \begin{aligned}
            \mathcal{M}(\delta_{\mathbf{v}_{[j,N]}}\brho,\delta_{\mathbf{v}_{[j,N]}}\bomega;\btau)=&(-1)^{(l+1)(d-1)}\int_{\Gamma^-}\star^\Gamma\mathbf{Tr}^\mathcal{D}(\delta_{\mathbf{v}_{[j,N]}}\bomega)\land\btau\\
            &+(-1)^{(l+1)(d-1)}\Big[\int_{\Gamma^-}\star\mathbf{i}_\mathbf{n}\star\mathbf{i}_\mathbf{n}\delta_{\mathbf{v}_{[j,N]}}\bomega\land\btau + \int_{\Gamma_R^+}\delta_{\mathbf{v}_{[j,N]}}\bomega\land\btau\Big]
        \end{aligned}
    \end{eqnarray}
    for $j = 1,\dots,N$. The integrals with integrands $\mathbf{i}_{\mathbf{v}_{[N+1]}}\btau$, $\mathbf{d}\btau$ and $\mathbf{i}_{\mathbf{v}_{[N+1]}}\mathbf{d}\btau$ can be eliminated by the similar identities as equations in ~\eqref{Betadefination}. 
    Then equation~\eqref{Eqnord} can be rewritten as
    \begin{eqnarray}
        \begin{aligned}
            \mathcal{M}(\delta_{\mathbf{v}_{[N+1]}}\brho,\delta_{\mathbf{v}_{[N+1]}}\bomega;\btau)
            =&(-1)^{(l+1)(d-1)}\Big[\int_{\Gamma^-}\star^\Gamma\mathbf{Tr}^{\mathcal{D}}(\delta_{\mathbf{v}_{[N+1]}}\bomega)\land\btau\\
            &+\int_{\Gamma^-}\star\mathbf{i}_\mathbf{n}\star\mathbf{i}_\mathbf{n}\delta_{\mathbf{v}_{[N+1]}}\bomega\land\btau +\int_{\Gamma_R^+}\delta_{\mathbf{v}_{[N+1]}}\bomega\land\btau\Big],
        \end{aligned}
    \end{eqnarray}
    where we denote
    \begin{eqnarray}\label{BoundNDiriTr}
        \begin{aligned}
             \star^\Gamma\mathbf{Tr}^{\mathcal{D}}(\delta_{\mathbf{v}_{[N+1]}}\bomega):=& \delta_{\mathbf{v}_{N+1}}^{\bomega}\big(\star^\Gamma\mathbf{Tr}^\mathcal{D}(\delta_{\mathbf{v}_{[N]}}\bomega)\big)+\mathbf{i}_{\mathbf{v}_{N+1}}\mathbf{d}\big(\star^\Gamma\mathbf{Tr}^\mathcal{D}(\delta_{\mathbf{v}_{[N]}}\bomega)\big)\\
        &-\mathbf{i}_{\mathbf{v}_{N+1}}\mathbf{d}\star^\Gamma\mathbf{i}_\mathbf{n}\star\delta_{\mathbf{v}_{[N]}}\bomega.
        \end{aligned}
    \end{eqnarray}
    It implies the $(N+1)$th order shape derivative $\delta_{\mathbf{v}_{[N+1]}}\bomega$ satisfies
    \begin{eqnarray}
    \left\{
        \begin{aligned}
            &\mathbf{d}\delta_{\mathbf{v}_{[N+1]}}\brho + (-1)^l\star_{k^2}\delta_{\mathbf{v}_{[N+1]}}\bomega = \mathbf{0}&&{\rm in}\quad\Omega^e,\\
            &\star_\alpha^{-1}\delta_{\mathbf{v}_{[N+1]}}\brho - (-1)^{(l+1)(d-1)}\mathbf{d}\delta_{\mathbf{v}_{[N+1]}}\bomega=\mathbf{0}&&{\rm in}\quad\Omega^e,\\
            &\star^\Gamma\mathbf{i}_{\mathbf{n}}\star\delta_{\mathbf{v}_{[N+1]}}\bomega = \star^\Gamma\mathbf{Tr}^{\mathcal{D}}(\delta_{\mathbf{v}_{[N+1]}}\bomega) &&{\rm on}\quad \Gamma,\\
            &\mathcal{F}(\delta_{\mathbf{v}_{[N+1]}}\brho,\delta_{\mathbf{v}_{[N+1]}}\bomega) = \mathbf{0}&&{\rm on}\quad \Gamma_R.
        \end{aligned}
        \right.
    \end{eqnarray}

    Equation~\eqref{BoundNDiriTr} provides the recurrence formula for the boundary condition on $\Gamma$ when transitioning from the $N$th to the $N+1$th order shape derivative under the velocity field~\eqref{veloField1}. For the general velocity field given by equation~\eqref{veloField2}, one must consider the perturbation on the normal vector $\mathbf{n}$ in equations~\eqref{SG1},~\eqref{Eq2ord} and~\eqref{Eqnord}. The corresponding material derivative~\eqref{SG1} is then replaced by
\begin{eqnarray}\label{SG2}
\begin{aligned}
    \delta_\mathbf{v}\mathcal{M}(\brho,\bomega;\btau) =& \mathcal{M}(\delta_\mathbf{v}\brho,\delta_\mathbf{v}\bomega;\btau) + \mathcal{L}_\mathbf{v}\mathcal{M}(\brho,\bomega;\btau)\\
    =& (-1)^{(l+1)(d-1)}\Big[\int_{\Gamma^-}\mathbf{i}_\mathbf{v}\mathbf{d}\big((\star^\Gamma\mathbf{Tr}^\mathcal{D}(\bphi) + \star\mathbf{i}_\mathbf{n}\star\mathbf{i}_\mathbf{n}\bomega)\land\btau\big) + \int_{\Gamma_R^+}\delta_\mathbf{v}\bomega\land\btau\Big]\\
    &+(-1)^{(l+1)(d-1)}\int_{\Gamma^-}\delta_\mathbf{v}^{\mathbf{n}}(\star^\Gamma\mathbf{Tr}^{\mathcal{D}}(\bphi) + \star\mathbf{i}_\mathbf{n}\star\mathbf{i}_\mathbf{n}\bomega)\land\btau,
\end{aligned}
\end{eqnarray}
which introduces a new term $\delta_\mathbf{v}^{\mathbf{n}}(\star^\Gamma\mathbf{Tr}^{\mathcal{D}}(\bphi) + \star\mathbf{i}_\mathbf{n}\star\mathbf{i}_\mathbf{n}\bomega)$.
Since the scattered field $\bomega$ restricted on $\Gamma$ is independent of the perturbation on the normal direction, it holds
\begin{eqnarray}
    \mathbf{0} = \delta_\mathbf{v}^{\mathbf{n}}(\bomega|_{\Gamma}) = \delta_\mathbf{v}^{\mathbf{n}}\star^\Gamma\mathbf{Tr}^{\mathcal{D}}(\bomega) + \delta_\mathbf{v}^{\mathbf{n}}\star\mathbf{i}_{\mathbf{n}}\star\mathbf{i}_{\mathbf{n}}\bomega.
\end{eqnarray}
By the identities 
\begin{eqnarray}
    \delta_\mathbf{v}^{\mathbf{n}}(\star^\Gamma\mathbf{Tr}^{\mathcal{D}}(\bphi) = \star^\Gamma\mathbf{i}_{\delta_\mathbf{v}\mathbf{n}}\star\bphi, \qquad\delta_\mathbf{v}^{\mathbf{n}}\star^\Gamma\mathbf{Tr}^{\mathcal{D}}(\bomega) = \star^\Gamma\mathbf{i}_{\delta_\mathbf{v}\mathbf{n}}\star\bomega,
\end{eqnarray}
the boundary condition on $\Gamma$ in equation~\eqref{Dbd1} becomes
    \begin{eqnarray}
        \star^\Gamma\mathbf{Tr}^{\mathcal{D}}(\delta_\mathbf{v}\bomega))=\mathbf{i}_\mathbf{v}\mathbf{d}\star^\Gamma\mathbf{i}_\mathbf{n}\star\bphi + \star^\Gamma\mathbf{i}_{\delta_\mathbf{v}\mathbf{n}}\star\bphi -\mathbf{i}_\mathbf{v}\mathbf{d}\star^\Gamma\mathbf{i}_\mathbf{n}\star\bomega - \star^\Gamma\mathbf{i}_{\delta_\mathbf{v}\mathbf{n}}\star\bomega.
    \end{eqnarray}
Similarly, the recurrence~\eqref{BoundNDiriTr} is also modified as
\begin{eqnarray}\label{recurrenceDriGen}
\begin{aligned}
        \star^\Gamma\mathbf{Tr}^{\mathcal{D}}(\delta_{\mathbf{v}_{[N+1]}}\bomega)=&\delta_{\mathbf{v}_{N+1}}^{\bomega}\big(\star^\Gamma\mathbf{Tr}^\mathcal{D}(\delta_{\mathbf{v}_{[N]}}\bomega)\big) + \delta_{\mathbf{v}_{N+1}}^{\mathbf{n}}\big(\star^\Gamma\mathbf{Tr}^\mathcal{D}(\delta_{\mathbf{v}_{[N]}}\bomega)\big)\\
        &+\mathbf{i}_{\mathbf{v}_{N+1}}\mathbf{d}\big(\star^\Gamma\mathbf{Tr}^\mathcal{D}(\delta_{\mathbf{v}_{[N]}}\bomega)\big)\\
        &-\mathbf{i}_{\mathbf{v}_{N+1}}\mathbf{d}\star^\Gamma\mathbf{i}_\mathbf{n}\star\delta_{\mathbf{v}_{[N]}}\bomega - \star^\Gamma\mathbf{i}_{\delta_{\mathbf{v}_{N+1}}\mathbf{n}}\star\delta_{\mathbf{v}_{[N]}}\bomega,
\end{aligned}
\end{eqnarray}
which completes the proof.
\end{proof}

\subsection{Neumann boundary}\label{NeumannBound}
For a scattering problem with Neumann boundary condition, also known as the sound hard boundary condition in acoustics and perfect magnetic conductor (PMC) condition in electromagnetics, one can rewrite equations~\eqref{EqD}-\eqref{EqD3} as
\begin{eqnarray}\label{Nbe}
\left\{
    \begin{aligned}
        &\mathbf{d}(\star_\alpha\mathbf{d}\bomega) + (-1)^l\star_{k^2}\bomega = \mathbf{0} &&{\rm in}\quad \Omega^e,\\
        &\mathbf{Tr}^\mathcal{N}(\bomega) = \mathbf{Tr}^\mathcal{N}(\bphi) = \star^\Gamma_\alpha\mathbf{i}_\mathbf{n}\mathbf{d}\bphi&&{\rm on}\quad\Gamma,\\
        &\mathcal{F}(\bomega,\star_\alpha \mathbf{d}\bomega) = \mathbf{0}&&{\rm on} \quad \Gamma_R,
    \end{aligned}
    \right.
\end{eqnarray}
where we denote $\mathbf{Tr}^{\mathcal{N}}: = \star^\Gamma_\alpha\mathbf{i}_{\mathbf{n}}\mathbf{d}$ the Neumann trace operator. It holds the following trace operator decomposition on $\Gamma$
\begin{eqnarray}\label{separation}
    \star_\alpha\mathbf{d}\bomega|_\Gamma = \star\mathbf{i}_\mathbf{n}\star\mathbf{i}_\mathbf{n}\star_\alpha\mathbf{d}\bomega + \mathbf{Tr}^\mathcal{N}(\bomega).
\end{eqnarray}
In particular, the decompositions for a scalar field $u$ and a vector field $\mathbf{E}$ on $\Gamma$ are given by
\begin{eqnarray}
\left\{
    \begin{aligned}
        &\alpha\nabla u|_\Gamma = -\mathbf{n}\times\mathbf{n}\times\alpha\nabla u + \mathbf{n}(\mathbf{n}\cdot\alpha\nabla u),\\
        &\alpha\nabla\times\mathbf{E}|_\Gamma = \mathbf{n} (\mathbf{n}\cdot\alpha\nabla\times\mathbf{E}) - \mathbf{n}\times\mathbf{n}\times\alpha\nabla\times\mathbf{E}.
    \end{aligned}
    \right.
\end{eqnarray}
The recurrence formula for the shape derivative of $\bomega$ under the Neumann boundary condition is given by the following theorem.
\begin{theorem}\label{ThemN}
    Let $\bomega$ be the solution of equation~\eqref{Nbe} and $\delta_{\mathbf{v}_{[N]}}\bomega$ be the $N$th order shape derivative w.r.t. $\mathbf{v}_{[N]}$ for $N = 0,1,\dots$, and  $\delta_{\mathbf{v}_{[N]}}\bomega = \bomega$ for $N = 0$. Then the $N+1$th order shape derivative $\delta_{\mathbf{v}_{[N+1]}}\bomega$ satisfies the same equation~\eqref{Nbe}, except that the boundary condition on $\Gamma$ is replaced by
     \begin{eqnarray}
         \begin{aligned}
         \mathbf{Tr}^\mathcal{N}(\delta_{\mathbf{v}_{[N+1]}}\bomega) =&\mathbf{i}_{\mathbf{v}_{N+1}}\mathbf{d}\mathbf{Tr}^{\mathcal{N}}(\delta_{\mathbf{v}_{[N]}}\bomega) + \delta_{\mathbf{v}_{N+1}}^{\bomega}\mathbf{Tr}^{\mathcal{N}}(\delta_{\mathbf{v}_{[N]}}\bomega) + \delta_{\mathbf{v}_{N+1}}^{\mathbf{n}}\mathbf{Tr}^{\mathcal{N}}(\delta_{\mathbf{v}_{[N]}}\bomega)\\
         &-\mathbf{i}_{\mathbf{v}_{N+1}}\mathbf{d}\star^\Gamma_\alpha\mathbf{i}_n\mathbf{d}\delta_{\mathbf{v}_{[N]}}\bomega - \star^\Gamma_\alpha\mathbf{i}_{\delta_{\mathbf{v}_{N+1}}\mathbf{n}}\mathbf{d}\delta_{\mathbf{v}_{[N]}}\bomega,
         \end{aligned}
     \end{eqnarray}
     where $\mathbf{Tr}^\mathcal{N}(\delta_{\mathbf{v}_{[N]}}\bomega)$ is the Neumann boundary condition for $\delta_{\mathbf{v}_{[N]}}\bomega$ on $\Gamma$. 
\end{theorem}

\begin{proof}

Given a test form $\btau\in\land^l$, the first equation in~\eqref{Nbe} yields
\begin{eqnarray}\label{Nbe_Int}
    \int_{\Omega^e}\mathbf{d}(\star_\alpha\mathbf{d}\bomega)\land\btau + (-1)^l\star_{k^2}\bomega\land\btau=0.
\end{eqnarray}
Based on integration by parts and Stokes' theorem, equation~\eqref{Nbe_Int} can be rewritten as 
\begin{eqnarray}\label{NeumannWeak}
\begin{aligned}
     &(-1)^{(d-l)}\int_{\Omega^e}\star_\alpha\mathbf{d}\bomega\land\mathbf{d}\btau + (-1)^l\int_{\Omega^e}\star_{k^2}\bomega\land\btau\\
      =& -\int_{\Gamma^-} \star_\alpha\mathbf{d}\bomega\land\btau - \int_{\Gamma_R^+}\star_\alpha\mathbf{d}\bomega\land\btau.
\end{aligned}
\end{eqnarray}
Denote $\mathcal{M}(\bomega;\btau)$ the left-hand side of equation~\eqref{NeumannWeak}. Substituting the decomposition~\eqref{separation} into the right-hand side of~\eqref{NeumannWeak} yields
\begin{eqnarray}\label{Mdv}
    \mathcal{M}(\bomega;\btau) = -\int_{\Gamma^-}\star\mathbf{i}_\mathbf{n}\star\mathbf{i}_\mathbf{n}\star_\alpha\mathbf{d}\bomega\land\btau -\int_{\Gamma^-} \mathbf{Tr}^{\mathcal{N}}(\bphi)\land\btau - \int_{\Gamma_R^+}\star_\alpha\mathbf{d}\bomega\land\btau.
\end{eqnarray}
Under the condition~\eqref{veloField1}, taking the material derivative w.r.t. $\mathbf{v}$ on both sides of~\eqref{Mdv} gives
\begin{eqnarray}\label{NBDmaterial}
\begin{aligned}
     \mathcal{M}(\delta_\mathbf{v}\bomega;\btau)+\mathcal{L}_\mathbf{v}\mathcal{M}(\bomega;\btau) =& -\int_{\Gamma^-} \mathbf{i}_\mathbf{v}\mathbf{d}(\star\mathbf{i}_\mathbf{n}\star\mathbf{i}_\mathbf{n}\star_\alpha\mathbf{d}\bomega\land\btau) - \int_{\Gamma^-}\mathbf{i}_\mathbf{v}\mathbf{d}(\mathbf{Tr}^{\mathcal{N}}(\bphi)\land\btau)\\
    &- \int_{\Gamma^-}\star\mathbf{i}_\mathbf{n}\star\mathbf{i}_\mathbf{n}\star_\alpha\mathbf{d}\delta_\mathbf{v}\bomega\land\btau - \int_{\Gamma_R^+}\star_\alpha\mathbf{d}\delta_\mathbf{v}\bomega\land\btau.
\end{aligned}
\end{eqnarray}
Based on the Lie derivative of $\mathcal{M}(\bomega;\btau)$, equation~\eqref{NBDmaterial} becomes
\begin{eqnarray}\label{NBD1}
\begin{aligned}
    \mathcal{M}(\delta_\mathbf{v}\bomega;\btau) =& -\int_{\Gamma^-} \big((-1)^l\mathbf{i}_\mathbf{v}\star_{k^2}\bomega + \mathbf{i}_\mathbf{v}\mathbf{d}\star\mathbf{i}_\mathbf{n}\star\mathbf{i}_\mathbf{n}\star_\alpha\mathbf{d}\bomega + \mathbf{i}_\mathbf{v}\mathbf{d}\mathbf{Tr}^{\mathcal{N}}(\bphi)\big)\land\btau\\
    &- (-1)^{d-l}\int_{\Gamma^-}\big((-1)^{l}\star_{k^2}\bomega + \mathbf{d}\star\mathbf{i}_\mathbf{n}\star\mathbf{i}_\mathbf{n}\star_\alpha\mathbf{d}\bomega + \mathbf{d}\mathbf{Tr}^{\mathcal{N}}(\bphi)\big)\land\mathbf{i}_\mathbf{v}\btau\\
    &-\int_{\Gamma^-}(-1)^{d-1}(\mathbf{i}_\mathbf{v}\star_\alpha\mathbf{d}\bomega - \mathbf{i}_\mathbf{v}\star\mathbf{i}_\mathbf{n}\star\mathbf{i}_\mathbf{n}\star_\alpha\mathbf{d}\bomega-\mathbf{i}_\mathbf{v}\mathbf{Tr}^{\mathcal{N}}(\bphi))\land\mathbf{d}\btau\\
    &-\int_{\Gamma^-}(\star\mathbf{i}_\mathbf{n}\star\mathbf{i}_\mathbf{n}\star_\alpha\mathbf{d}\bomega + \mathbf{Tr}^{\mathcal{N}}(\bphi) -\star_\alpha\mathbf{d}\bomega)\land\mathbf{i}_\mathbf{v}\mathbf{d}\btau\\
    &- \int_{\Gamma^-}\star\mathbf{i}_\mathbf{n}\star\mathbf{i}_\mathbf{n}\star_\alpha\mathbf{d}\delta_\mathbf{v}\bomega\land\btau - \int_{\Gamma_R^+}\star_\alpha\mathbf{d}\delta_\mathbf{v}\bomega\land\btau.
\end{aligned}
\end{eqnarray}
 Employing the following identities on $\Gamma$: 
 \begin{eqnarray}
 \left\{
 \begin{aligned}
     &(-1)^l\star_{k^2}\bomega + \mathbf{d}\star\mathbf{i}_\mathbf{n}\star\mathbf{i}_\mathbf{n}\star_\alpha\mathbf{d}\bomega = - \mathbf{d}\star^\Gamma_\alpha\mathbf{i}_\mathbf{n}\mathbf{d}\bomega,\\
     &\mathbf{i}_{\mathbf{n}}\star\Big(\big( \mathbf{d}\mathbf{Tr}^{\mathcal{N}}(\bphi) -  \mathbf{d}\star^\Gamma_\alpha\mathbf{i}_\mathbf{n}\mathbf{d}\bomega\big)\land\mathbf{i}_\mathbf{v}\btau\Big) = \mathbf{0},\\
     &\star\mathbf{i}_\mathbf{n}\star\mathbf{i}_\mathbf{n}\star_\alpha\mathbf{d}\bomega + \mathbf{Tr}^{\mathcal{N}}(\bphi) - \star_\alpha\mathbf{d}\bomega = \mathbf{0},
 \end{aligned}
 \right.
 \end{eqnarray}
 equation~\eqref{NBD1} can be rewritten as
 \begin{eqnarray}\label{NBD11}
 \begin{aligned}
     \mathcal{M}(\delta_\mathbf{v}\bomega;\btau)=&-\int_{\Gamma^-} (-\mathbf{i}_\mathbf{v}\mathbf{d}\star^\Gamma_\alpha\mathbf{i}_\mathbf{n}\mathbf{d}\bomega + \mathbf{i}_\mathbf{v}\mathbf{d}\mathbf{Tr}^{\mathcal{N}}(\bphi))\land\btau \\
    &-\int_{\Gamma^-}\star\mathbf{i}_\mathbf{n}\star\mathbf{i}_\mathbf{n}\star_\alpha\mathbf{d}\delta_\mathbf{v}\bomega\land\btau - \int_{\Gamma_R^+}\star_\alpha\mathbf{d}\delta_\mathbf{v}\bomega\land\btau.
    \end{aligned}
 \end{eqnarray}
 Equation~\eqref{NBD11} implies that the first order shape derivative $\delta_\mathbf{v}\bomega$ satisfies the equation
 \begin{eqnarray}\label{Nbd1e}
 \left\{
     \begin{aligned}
         &\mathbf{d}(\star_\alpha\mathbf{d}\delta_\mathbf{v}\bomega) + (-1)^l\star_{k^2}\delta_\mathbf{v}\bomega = \mathbf{0} &&{\rm in}\quad \Omega^e,\\
         &\mathbf{Tr}^\mathcal{N}(\delta_\mathbf{v}\bomega) = -\mathbf{i}_\mathbf{v}\mathbf{d}\star^\Gamma_\alpha\mathbf{i}_\mathbf{n}\mathbf{d}\bomega + \mathbf{i}_\mathbf{v}\mathbf{d}\mathbf{Tr}^{\mathcal{N}}(\bphi) &&{\rm on}\quad\Gamma,\\
         &\mathcal{F}(\delta_\mathbf{v}\bomega,\star_\alpha \mathbf{d}\delta_\mathbf{v}\bomega) = \mathbf{0} &&{\rm on} \quad \Gamma_R.
     \end{aligned}
     \right.
 \end{eqnarray}

Following the same procedure as the Dirichlet case, the second order shape derivative $\delta_{[\mathbf{v},\mathbf{w}]}\bomega$ can be obtained by taking material derivative of equation~\eqref{NBD11} w.r.t. the second velocity field $\mathbf{w}$:
\begin{eqnarray}\label{NeumannWeak2}
    \begin{aligned}
        &\mathcal{M}(\delta_{[\mathbf{v},\mathbf{w}]}\bomega;\btau) + \mathcal{L}_\mathbf{w}\mathcal{M}(\delta_\mathbf{v}\bomega,\btau)\\
        =&-\int_{\Gamma^-}\mathbf{i}_\mathbf{w}\mathbf{d}\Big((-\mathbf{i}_\mathbf{v}\mathbf{d}\star^\Gamma_\alpha\mathbf{i}_\mathbf{n}\mathbf{d}\bomega + \mathbf{i}_\mathbf{v}\mathbf{d}\mathbf{Tr}^{\mathcal{N}}(\bphi))\land\btau\Big)\\
        &-\int_{\Gamma^-}\mathbf{i}_\mathbf{w}\mathbf{d}(\star\mathbf{i}_\mathbf{n}\star\mathbf{i}_\mathbf{n}\star_\alpha\mathbf{d}\delta_\mathbf{v}\bomega\land\btau)\\
        &-\int_{\Gamma^-} \big((-1)^l\mathbf{i}_\mathbf{v}\star_{k^2}\delta_\mathbf{w}\bomega + \mathbf{i}_\mathbf{v}\mathbf{d}\star\mathbf{i}_\mathbf{n}\star\mathbf{i}_\mathbf{n}\star_\alpha\mathbf{d}\delta_\mathbf{w}\bomega\big)\land\btau\\
        &- \int_{\Gamma^-}\star\mathbf{i}_\mathbf{n}\star\mathbf{i}_\mathbf{n}\star_\alpha\mathbf{d}\delta_{[\mathbf{v},\mathbf{w}]}\bomega\land\btau - \int_{\Gamma_R^+}\star_\alpha\mathbf{d}\delta_{[\mathbf{v},\mathbf{w}]}\bomega\land\btau.\\
    \end{aligned}
\end{eqnarray}
Here we omit the terms containing $\delta_{\mathbf{w}}\mathbf{v}$ in equation~\eqref{NeumannWeak2} by using the same argument as in equation~\eqref{Eliminate}. According to the following identities on $\Gamma$:
\begin{eqnarray}\label{Nbdw3}
\left\{
    \begin{aligned}
    &(-1)^l\star_{k^2}\delta_\mathbf{v}\bomega + \mathbf{d}\star\mathbf{i}_\mathbf{n}\star\mathbf{i}_\mathbf{n}\star_\alpha\mathbf{d}\delta_\mathbf{v}\bomega = - \mathbf{d}\star^\Gamma_\alpha\mathbf{i}_\mathbf{n}\mathbf{d}\delta_\mathbf{v}\bomega,\\
     &\mathbf{i}_{\mathbf{n}}\star\Big(\big( \mathbf{d}\mathbf{Tr}^{\mathcal{N}}(\delta_\mathbf{v}\bomega) -  \mathbf{d}\star^\Gamma_\alpha\mathbf{i}_\mathbf{n}\mathbf{d}\delta_\mathbf{v}\bomega\big)\land\mathbf{i}_\mathbf{w}\btau\Big) = \mathbf{0},\\
     &\star\mathbf{i}_\mathbf{n}\star\mathbf{i}_\mathbf{n}\star_\alpha\mathbf{d}\delta_\mathbf{v}\bomega + \mathbf{Tr}^{\mathcal{N}}(\delta_\mathbf{v}\bomega) - \star_\alpha\mathbf{d}\delta_\mathbf{v}\bomega = \mathbf{0},
    \end{aligned}
\right.
\end{eqnarray}
equation~\eqref{NeumannWeak2} can be rewritten as 
\begin{eqnarray}\label{Nbdw1}
\begin{aligned}
    \mathcal{M}(\delta_{[\mathbf{v},\mathbf{w}]}\bomega,\btau) =& -\int_{\Gamma^-}\mathbf{Tr}^\mathcal{N}(\delta_{[\mathbf{v},\mathbf{w}]}\bomega)\land\btau \\
    &- \int_{\Gamma^-}\star\mathbf{i}_\mathbf{n}\star\mathbf{i}_\mathbf{n}\star_\alpha\mathbf{d}\delta_{[\mathbf{v},\mathbf{w}]}\bomega\land\btau - \int_{\Gamma_R^+}\star_\alpha\mathbf{d}\delta_{[\mathbf{v},\mathbf{w}]}\bomega\land\btau,
\end{aligned}
\end{eqnarray}
where the trace operator $\mathbf{Tr}^\mathcal{N}(\delta_{[\mathbf{v},\mathbf{w}]}\bomega)$ is given by
\begin{eqnarray}\label{Nbdw2}
    \begin{aligned}
        \mathbf{Tr}^\mathcal{N}(\delta_{[\mathbf{v},\mathbf{w}]}\bomega) =& -\mathbf{i}_\mathbf{w}\mathbf{d}\mathbf{i}_\mathbf{v}\mathbf{d}\star^\Gamma_\alpha\mathbf{i}_\mathbf{n}\mathbf{d}\bomega + \mathbf{i}_\mathbf{w}\mathbf{d}\mathbf{i}_\mathbf{v}\mathbf{d}\mathbf{Tr}^{\mathcal{N}}(\bphi)\\
        & + (-1)^l\mathbf{i}_\mathbf{v}\star_{k^2}\delta_\mathbf{w}\bomega + \mathbf{i}_\mathbf{v}\mathbf{d}\star\mathbf{i}_\mathbf{n}\star\mathbf{i}_\mathbf{n}\star_\alpha\mathbf{d}\delta_\mathbf{w}\bomega\\
        & + (-1)^l\mathbf{i}_\mathbf{w}\star_{k^2}\delta_\mathbf{v}\bomega + \mathbf{i}_\mathbf{w}\mathbf{d}\star\mathbf{i}_\mathbf{n}\star\mathbf{i}_\mathbf{n}\star_\alpha\mathbf{d}\delta_\mathbf{v}\bomega\\
        =&-\mathbf{i}_\mathbf{w}\mathbf{d}\mathbf{i}_\mathbf{v}\mathbf{d}\star^\Gamma_\alpha\mathbf{i}_\mathbf{n}\mathbf{d}\bomega + \mathbf{i}_\mathbf{w}\mathbf{d}\mathbf{i}_\mathbf{v}\mathbf{d}\mathbf{Tr}^{\mathcal{N}}(\bphi)\\
        &- \mathbf{i}_\mathbf{v}\mathbf{d}\star^\Gamma_\alpha\mathbf{i}_\mathbf{n}\mathbf{d}\delta_\mathbf{w}\bomega -\mathbf{i}_\mathbf{w}\mathbf{d}\star^\Gamma_\alpha\mathbf{i}_\mathbf{n}\mathbf{d}\delta_\mathbf{v}\bomega.
    \end{aligned}
\end{eqnarray}
Therefore, the equation for the second order shape derivative is formulated as
\begin{eqnarray}\label{NeumBD2}
\left\{
\begin{aligned}
    &\mathbf{d}(\star_\alpha\mathbf{d}\delta_{[\mathbf{v},\mathbf{w}]}\bomega) + (-1)^l\star_{k^2}\delta_{[\mathbf{v},\mathbf{w}]}\bomega = \mathbf{0} &&{\rm in}\quad \Omega^e,\\
    &\star^\Gamma_\alpha\mathbf{i}_{\mathbf{n}}\mathbf{d}\delta_{[\mathbf{v},\mathbf{w}]} = \mathbf{Tr}^N(\delta_{[\mathbf{v},\mathbf{w}]}\bomega)&&{\rm on}\quad\Gamma,\\
    &\mathcal{F}(\delta_{[\mathbf{v},\mathbf{w}]}\bomega,\star_\alpha \mathbf{d}\delta_{[\mathbf{v},\mathbf{w}]}\bomega) = \mathbf{0}&&{\rm on} \quad \Gamma_R.
\end{aligned}
\right.
\end{eqnarray}

We inductively prove the boundary conditions for higher order shape derivatives w.r.t. the velocity fields $\mathbf{v}_{[N+1]}$. Suppose that the boundary condition for $\delta_{\mathbf{v}_{[N]}}\bomega$ has been given by $\mathbf{Tr}^{\mathcal{N}}(\delta_{\mathbf{v}_{[N]}}\bomega)$. Then $\delta_{\mathbf{v}_{[N+1]}}\bomega$ satisfies
\begin{eqnarray}\label{NbdwN0}
    \begin{aligned}
         &\mathcal{M}(\delta_{\mathbf{v}_{[N+1]}}\bomega,\btau)+\mathcal{L}_{\mathbf{v}_{N+1}}\mathcal{M}(\delta_{\mathbf{v}_{[N]}}\bomega,\btau) \\=
        &-\int_{\Gamma^-}\mathbf{i}_{\mathbf{v}_{N+1}}\mathbf{d}\big(\mathbf{Tr}^\mathcal{N}(\delta_{\mathbf{v}_{[N]}}\bomega)\land\btau\big)-\int_{\Gamma^-}\mathbf{i}_{\mathbf{v}_{N+1}}\mathbf{d}(\star\mathbf{i}_\mathbf{n}\star\mathbf{i}_\mathbf{n}\star_\alpha\mathbf{d}\delta_{\mathbf{v}_{[N]}}\bomega\land\btau)\\
        &-\int_{\Gamma^-} \delta_{\mathbf{v}_{N+1}}^{\bomega}\mathbf{Tr}^\mathcal{N}(\delta_{\mathbf{v}_{[N]}}\bomega)\land\btau- \int_{\Gamma^-}\star\mathbf{i}_\mathbf{n}\star\mathbf{i}_\mathbf{n}\star_\alpha\mathbf{d}\delta_{\mathbf{v}_{[N+1]}}\bomega\land\btau\\
        &- \int_{\Gamma_R^+}\star_\alpha\mathbf{d}\delta_{\mathbf{v}_{[N+1]}}\bomega\land\btau.\\
    \end{aligned}
\end{eqnarray}
According to the definition of Lie derivative for $\mathcal{M}(\delta_{\mathbf{v}_{[N]}}\bomega,\btau)$ and the following identities on $\Gamma$
\begin{eqnarray}\label{NbdwN2}
\left\{
\begin{aligned}
    &(-1)^l\star_{k^2}\delta_{\mathbf{v}_{[N]}}\bomega + \mathbf{d}\star\mathbf{i}_\mathbf{n}\star\mathbf{i}_\mathbf{n}\star_\alpha\mathbf{d}\delta_{\mathbf{v}_{[N]}}\bomega = - \mathbf{d}\star^\Gamma_\alpha\mathbf{i}_\mathbf{n}\mathbf{d}\delta_{\mathbf{v}_{[N]}}\bomega,\\
     &\mathbf{i}_{\mathbf{n}}\star\Big(\big( \mathbf{d}\mathbf{Tr}^{\mathcal{N}}(\delta_{\mathbf{v}_{[N]}}\bomega) -  \mathbf{d}\star^\Gamma_\alpha\mathbf{i}_\mathbf{n}\mathbf{d}\delta_{\mathbf{v}_{[N]}}\bomega\big)\land\mathbf{i}_{\mathbf{v}_{N+1}}\btau\Big) = \mathbf{0},\\
     &\star\mathbf{i}_\mathbf{n}\star\mathbf{i}_\mathbf{n}\star_\alpha\mathbf{d}\delta_{\mathbf{v}_{[N]}}\bomega + \mathbf{Tr}^{\mathcal{N}}(\delta_{\mathbf{v}_{[N]}}\bomega) - \star_\alpha\mathbf{d}\delta_{\mathbf{v}_{[N]}}\bomega = \mathbf{0},
\end{aligned}
\right.
\end{eqnarray}
equation~\eqref{NbdwN0} can be rewritten as
\begin{eqnarray}\label{NbdwN1}
\begin{aligned}
    \mathcal{M}(\delta_{\mathbf{v}_{[N+1]}}\bomega,\btau) =&-\int_{\Gamma^-} \mathbf{Tr}^{\mathcal{N}}(\delta_{[N+1]}\bomega)\land\btau\\
    &- \int_{\Gamma^-}\star\mathbf{i}_\mathbf{n}\star\mathbf{i}_\mathbf{n}\star_\alpha\mathbf{d}\delta_{\mathbf{v}_{[N+1]}}\bomega\land\btau - \int_{\Gamma_R^+}\star_\alpha\mathbf{d}\delta_{\mathbf{v}_{[N+1]}}\bomega\land\btau,
\end{aligned}
\end{eqnarray}
with
\begin{eqnarray}
\begin{aligned}
    \mathbf{Tr}^{\mathcal{N}}(\delta_{[N+1]}\bomega) =& \mathbf{i}_{\mathbf{v}_{N+1}} \mathbf{d}\mathbf{Tr}^{\mathcal{N}}(\delta_{\mathbf{v}_{[N]}}\bomega) + \delta_{\mathbf{v}_{N+1}}^{\bomega}\mathbf{Tr}^{\mathcal{N}}(\delta_{\mathbf{v}_{[N]}}\bomega)\\
    &+\mathbf{i}_{\mathbf{v}_{N+1}}\mathbf{d}\star\mathbf{i}_\mathbf{n}\star\mathbf{i}_\mathbf{n}\star_\alpha\mathbf{d}\delta_{\mathbf{v}_{[N]}}\bomega - \mathbf{i}_{\mathbf{v}_{N+1}}\mathbf{d}\star_\alpha\mathbf{d}\delta_{\mathbf{v}_{[N]}}\bomega\\
    =&\mathbf{i}_{\mathbf{v}_{N+1}}\mathbf{d}\mathbf{Tr}^{\mathcal{N}}(\delta_{\mathbf{v}_{[N]}}\bomega) + \delta_{\mathbf{v}_{N+1}}\mathbf{Tr}^{\mathcal{N}}(\delta_{\mathbf{v}_{[N]}}\bomega) -\mathbf{i}_{\mathbf{v}_{N+1}}\mathbf{d}\star^\Gamma_\alpha\mathbf{i}_\mathbf{n}\mathbf{d}\delta_{\mathbf{v}_{[N]}}\bomega.
\end{aligned}
\end{eqnarray}
Therefore, the equation for the $(N+1)$th order shape derivative is given by
\begin{eqnarray}\label{Neunord}
\left\{
\begin{aligned}
     &\mathbf{d}(\star_\alpha\mathbf{d}\delta_{\mathbf{v}_{[N+1]}}\bomega) + (-1)^l\star_{k^2}\delta_{\mathbf{v}_{[N+1]}}\bomega = \mathbf{0}&& {\rm in}\quad \Omega^e,\\
    &\star^\Gamma_\alpha\mathbf{i}_{\mathbf{n}}\mathbf{d}\delta_{\mathbf{v}_{[N+1]}}\bomega = \mathbf{Tr}^\mathcal{N}(\delta_{\mathbf{v}_{[N+1]}}\bomega) &&{\rm on}\quad\Gamma,\label{reccforNeumann}\\
    &\mathcal{F}(\delta_{\mathbf{v}_{[N+1]}}\bomega,\star_\alpha \mathbf{d}\delta_{\mathbf{v}_{[N+1]}}\bomega) = \mathbf{0}&& {\rm on}\quad \Gamma_R.
\end{aligned}
\right.
\end{eqnarray}

Note that equation~\eqref{Neunord} is derived under the velocity fields of the form given in equation~\eqref{veloField1}. For the general case given by equation~\eqref{veloField2}, equation~\eqref{NBDmaterial} is replaced by:
\begin{eqnarray}\label{NBDmaterial2}
\begin{aligned}
    &\mathcal{M}(\delta_\mathbf{v}\bomega;\btau)+\mathcal{L}_\mathbf{v}\mathcal{M}(\bomega;\btau)\\
    =& -\int_{\Gamma^-} \mathbf{i}_\mathbf{v}\mathbf{d}(\star\mathbf{i}_\mathbf{n}\star\mathbf{i}_\mathbf{n}\star_\alpha\mathbf{d}\bomega\land\btau) - \int_{\Gamma^-}\mathbf{i}_\mathbf{v}\mathbf{d}(\mathbf{Tr}^{\mathcal{N}}(\bphi)\land\btau)\\
    &-\int_{\Gamma^-}\star\mathbf{i}_\mathbf{n}\star\mathbf{i}_\mathbf{n}\star_\alpha\mathbf{d}\delta_\mathbf{v}\bomega\land\btau -\int_{\Gamma^-}\delta_\mathbf{v}^{\mathbf{n}}(\star\mathbf{i}_\mathbf{n}\star\mathbf{i}_\mathbf{n}\star_\alpha\mathbf{d}\bomega)\land\btau\\
    &- \int_{\Gamma^-}\delta_\mathbf{v}^\mathbf{n}(\star^\Gamma_\alpha\mathbf{i}_\mathbf{n}\mathbf{d}\phi)\land\btau-\int_{\Gamma_R^+}\star_\alpha\mathbf{d}\delta_\mathbf{v}\bomega\land\btau.
\end{aligned}
\end{eqnarray}
Since $\star_\alpha\mathbf{d}\bomega$ restricted on $\Gamma$ is independent of the normal direction, one has
\begin{eqnarray}
    \mathbf{0} = \delta_\mathbf{v}^{\mathbf{n}}(\star_\alpha\mathbf{d}\bomega|_\Gamma) = \delta_\mathbf{v}^{\mathbf{n}}(\star\mathbf{i}_\mathbf{n}\star\mathbf{i}_\mathbf{n}\star_\alpha\mathbf{d}\bomega) + \delta_\mathbf{v}^{\mathbf{n}}\mathbf{Tr}^\mathcal{N}(\bomega) \qquad {\rm on}\quad\Gamma.
\end{eqnarray}
Thus, the boundary condition on $\Gamma$ for the first order shape derivative $\delta_{\mathbf{v}}\bomega$ is modified to
\begin{eqnarray}\label{traced2}
\begin{aligned}
    \mathbf{Tr}^\mathcal{N}(\delta_\mathbf{v}\bomega) =& -\mathbf{i}_\mathbf{v}\mathbf{d}\star^\Gamma_\alpha\mathbf{i}_\mathbf{n}\mathbf{d}\bomega - \star^\Gamma_\alpha\mathbf{i}_{\delta_\mathbf{v}\mathbf{n}}\mathbf{d}\bomega \\
    &+ \mathbf{i}_\mathbf{v}\mathbf{d}\mathbf{Tr}^{\mathcal{N}}(\bphi) + \star^\Gamma_\alpha\mathbf{i}_{\delta_\mathbf{v}\mathbf{n}}\mathbf{d}\bphi.
\end{aligned}
\end{eqnarray}
 Analogously, the recurrence formula in~\eqref{reccforNeumann} for the $N+1$th order shape derivative is replaced by
 \begin{eqnarray}\label{tracerecurrence2}
 \begin{aligned}
     \mathbf{Tr}^\mathcal{N}(\delta_{\mathbf{v}_{[N+1]}}\bomega) =&\mathbf{i}_{\mathbf{v}_{N+1}}\mathbf{d}\mathbf{Tr}^{\mathcal{N}}(\delta_{\mathbf{v}_{[N]}}\bomega) + \delta_{\mathbf{v}_{N+1}}^{\bomega}\mathbf{Tr}^{\mathcal{N}}(\delta_{\mathbf{v}_{[N]}}\bomega) + \delta_{\mathbf{v}_{N+1}}^{\mathbf{n}}\mathbf{Tr}^{\mathcal{N}}(\delta_{\mathbf{v}_{[N]}}\bomega)\\
    &-\mathbf{i}_{\mathbf{v}_{N+1}}\mathbf{d}\star^\Gamma_\alpha\mathbf{i}_n\mathbf{d}\delta_{\mathbf{v}_{[N]}}\bomega - \star^\Gamma_\alpha\mathbf{i}_{\delta_{\mathbf{v}_{N+1}}\mathbf{n}}\mathbf{d}\delta_{\mathbf{v}_{[N]}}\bomega.
 \end{aligned}
 \end{eqnarray}
 \end{proof}
\section{Extension to impedance and transmission boundary conditions}\label{sec4}
The derivations of the shape derivatives for the impedance and transmission boundary conditions closely follow the procedures used for the Dirichlet and Neumann cases. In this section, we provide a brief derivation of the shape derivatives for these two boundary conditions.
\subsection{Impedance boundary}\label{ImpBound}
Consider the total field $\bomega$ on $\Gamma$  that satisfies the impedance boundary condition
\begin{eqnarray}\label{impbound}
    \mathbf{Tr}^\mathcal{N}(\bomega) + (-1)^l i\lambda\mathbf{Tr}^\mathcal{D}(\bomega) = \mathbf{0}.
\end{eqnarray}
It is important to note that there is no essential difference between the shape derivatives of the total field and the scattered field, as the incident field remains unaffected by the perturbations of $\Gamma$. The recurrence formula of shape derivatives for $\bomega$ under the impedance boundary condition is given by the following theorem.
\begin{theorem}\label{TheoI}
    Let $\bomega$ be the solution of equation~\eqref{Nbe} with boundary condition~\eqref{impbound} on $\Gamma$ and $\delta_{\mathbf{v}_{[N]}}\bomega$ be the $N$th order shape derivative w.r.t. $\mathbf{v}_{[N]}$, with $\delta_{\mathbf{v}_{[N]}}\bomega=\bomega$ for $N=0$. Then the recurrence formula for the boundary condition on $\Gamma$ of the $N+1$th shape derivative $\delta_{\mathbf{v}_{[N+1]}}\bomega$ on $\Gamma$ is given by
    \begin{eqnarray}\label{tracerecurrence32}
    \begin{aligned}
    &\mathbf{Tr}^\mathcal{N}(\delta_{\mathbf{v}_{[N+1]}}\bomega) + (-1)^li\lambda\mathbf{Tr}^\mathcal{D}(\delta_{\mathbf{v}_{[N+1]}}\bomega)\\
    =& -\mathbf{i}_{\mathbf{v}_{N+1}}\mathbf{d}\big(\mathbf{Tr}^\mathcal{N}(\delta_{\mathbf{v}_{[N]}}\bomega) + (-1)^li\lambda\mathbf{Tr}^\mathcal{D}(\delta_{\mathbf{v}_{[N]}}\bomega)\big)\\
    &-\delta_{\mathbf{v}_{N+1}}^{\bomega}\big(\mathbf{Tr}^\mathcal{N}(\delta_{\mathbf{v}_{[N]}}\bomega) + (-1)^li\lambda\mathbf{Tr}^\mathcal{D}(\delta_{\mathbf{v}_{[N]}}\bomega)\big)\\
    &-\delta_{N+1}^{\mathbf{n}}\big(\mathbf{Tr}^\mathcal{N}(\delta_{\mathbf{v}_{[N]}}\bomega) + (-1)^li\lambda\mathbf{Tr}^\mathcal{D}(\delta_{\mathbf{v}_{[N]}}\bomega)\big)\\
    &-\mathbf{i}_{\mathbf{v}_{N+1}}\mathbf{d}\big(\star^\Gamma_\alpha\mathbf{i}_\mathbf{n}\mathbf{d}\delta_{\mathbf{v}_{[N]}}\bomega + (-1)^li\lambda\mathbf{i}_\mathbf{n}\star\delta_{\mathbf{v}_{[N]}}\bomega\big)\\
    &-\delta_{\mathbf{v}_{N+1}}^{\mathbf{n}}\big(\star^\Gamma_\alpha\mathbf{i}_\mathbf{n}\mathbf{d}\delta_{\mathbf{v}_{[N]}}\bomega + (-1)^li\lambda\mathbf{i}_\mathbf{n}\star\delta_{\mathbf{v}_{[N]}}\bomega\big).
    \end{aligned}
    \end{eqnarray}
\end{theorem}

\begin{proof}

For simplicity, we still begin by assuming that the velocity field $\mathbf{v}$ is given by equation~\eqref{veloField1}. Replacing the term $\mathbf{Tr}^{\mathcal{N}}(\bphi)$ in equation~\eqref{Mdv} by $-(-1)^li\lambda\mathbf{Tr}^\mathcal{D}(\bomega)$, it yields
\begin{eqnarray}\label{impe0}
    \mathcal{M}(\bomega,\btau) = - \int_{\Gamma^-}\star\mathbf{i}_\mathbf{n}\star\mathbf{i}_\mathbf{n}\star_\alpha\mathbf{d}\bomega\land\btau - (-1)^li\lambda\mathbf{Tr}^\mathcal{D}(\bomega)\land\btau - \int_{\Gamma_R^+}\star_\alpha\mathbf{d}\bomega\land\btau.
\end{eqnarray}
Take the material derivative of equation~\eqref{impe0}. Based on the same argument as in the Neumann case, we obtain
\begin{eqnarray}\label{Imp01}
    \begin{aligned}
        \mathcal{M}(\delta_\mathbf{v}\bomega;\btau) =& -\int_{\Gamma^-}(-1)^{d-l}\mathbf{i}_\mathbf{v}(\star_\alpha\mathbf{d}\bomega\land\mathbf{d}\btau) + (-1)^{l}\mathbf{i}_\mathbf{v}(\star_{k^2}\bomega\land\btau)\\
        &-\int_{\Gamma^-}\mathbf{i}_\mathbf{v}\mathbf{d}(\star\mathbf{i}_\mathbf{n}\star\mathbf{i}_\mathbf{n}\star_\alpha\mathbf{d}\bomega\land\btau) - \mathbf{i}_\mathbf{v}\mathbf{d}(i\lambda\mathbf{Tr}^D(\bomega)\land\btau)\\
        &-\int_{\Gamma^-}\star\mathbf{i}_\mathbf{n}\star\mathbf{i}_\mathbf{n}\star_\alpha\mathbf{d}\delta_\mathbf{v}\bomega\land\btau - (-1)^li\lambda\mathbf{Tr}^\mathcal{D}(\delta_\mathbf{v}\bomega)\land\btau - \int_{\Gamma_R^+}\star_\alpha\delta_\mathbf{v}\bomega\land\btau\\
        =&-\int_{\Gamma^-} - \mathbf{i}_\mathbf{v}\mathbf{d}\big(\mathbf{Tr}^\mathcal{N}(\bomega) + (-1)^li\lambda\mathbf{Tr}^\mathcal{D}(\bomega)\big)\land\btau - (-1)^li\lambda\mathbf{Tr}^\mathcal{D}(\delta_\mathbf{v}\bomega)\land\btau\\
        &-\int_{\Gamma^-}\star\mathbf{i}_\mathbf{n}\star\mathbf{i}_\mathbf{n}\star_\alpha\mathbf{d}\delta_\mathbf{v}\bomega\land\btau - \int_{\Gamma_R^+}\star_\alpha\mathbf{d}\bomega\land\btau.
    \end{aligned}
\end{eqnarray}
Thus the boundary condition of $\delta_\mathbf{v}\bomega$ on $\Gamma$ is given by
\begin{eqnarray}\label{impB01}
    \mathbf{Tr}^\mathcal{N}(\delta_\mathbf{v}\bomega) + (-1)^li\lambda\mathbf{Tr}^\mathcal{D}(\delta_\mathbf{v}\bomega) = -(-1)^l\mathbf{i}_\mathbf{v}\mathbf{d}\big(\mathbf{Tr}^\mathcal{N}(\bomega) + (-1)^li\lambda\mathbf{Tr}^\mathcal{D}(\bomega)\big).
\end{eqnarray}
Analogously, the recurrence formula from the $N$th order shape derivative to the $(N+1)$th order is given by
\begin{eqnarray}\label{tracerecurrence3}
\begin{aligned}
    &\mathbf{Tr}^\mathcal{N}(\delta_{\mathbf{v}_{[N+1]}}\bomega) + (-1)^li\lambda\mathbf{Tr}^\mathcal{D}(\delta_{\mathbf{v}_{[N+1]}}\bomega)\\
    =& -\mathbf{i}_{\mathbf{v}_{N+1}}\mathbf{d}\big(\mathbf{Tr}^\mathcal{N}(\delta_{\mathbf{v}_{[N]}}\bomega) + (-1)^li\lambda\mathbf{Tr}^\mathcal{D}(\delta_{\mathbf{v}_{[N]}}\bomega)\big)\\
    &-\delta_{\mathbf{v}_{N+1}}^{\bomega}\big(\mathbf{Tr}^\mathcal{N}(\delta_{\mathbf{v}_{[N]}}\bomega) + (-1)^li\lambda\mathbf{Tr}^\mathcal{D}(\delta_{\mathbf{v}_{[N]}}\bomega)\big)\\
    &-\mathbf{i}_{\mathbf{v}_{N+1}}\mathbf{d}\big(\star^\Gamma_\alpha\mathbf{i}_\mathbf{n}\mathbf{d}\delta_{\mathbf{v}_{[N]}}\bomega + (-1)^li\lambda\mathbf{i}_\mathbf{n}\delta_{\mathbf{v}_{[N]}}\bomega\big).
\end{aligned}
\end{eqnarray}
 For general perturbations defined by equation~\eqref{veloField2}, the proof is completely the same as in the Neumann case. The boundary condition on $\Gamma$ given by~\eqref{impB01} is replaced by
\begin{eqnarray}
\begin{aligned}
    \mathbf{Tr}^\mathcal{N}(\delta_\mathbf{v}\bomega) + (-1)^li\lambda\mathbf{Tr}^\mathcal{D}(\delta_\mathbf{v}\bomega) =& -\mathbf{i}_\mathbf{v}\mathbf{d}\big(\mathbf{Tr}^\mathcal{N}(\bomega) + (-1)^li\lambda\mathbf{Tr}^\mathcal{D}(\bomega)\big)\\
    &- \delta^{\mathbf{n}}_{\mathbf{v}_{N+1}}\big(\mathbf{Tr}^\mathcal{N}(\bomega) + (-1)^li\lambda\mathbf{Tr}^\mathcal{D}(\bomega)\big),
\end{aligned}
\end{eqnarray}
and the recurrence formula~\eqref{tracerecurrence32} is obtained based on the same derivation as equation~\eqref{tracerecurrence2}.

\end{proof}

\subsection{Transmission boundary}\label{TransBound}
The transmission boundary condition on $\Gamma$ consists of $[\star^\Gamma\mathbf{Tr}^\mathcal{D}(\bomega)]$ and $[\mathbf{Tr}^\mathcal{N}(\bomega)]$, where $[\cdot]$ represents the jump across $\Gamma$ defined by
\begin{eqnarray}
    [f(\mathbf{x})] = \lim_{\mathbf{x}^e\in\Omega^e,\mathbf{x}^e\rightarrow \mathbf{x}}f(\mathbf{x}^e) - \lim_{\mathbf{x}^i\in\Omega,\mathbf{x}^i\rightarrow \mathbf{x}}f(\mathbf{x}^i), \quad \mathbf{x}\in\Gamma.
\end{eqnarray}
Here the total field $\bomega$ is an $l$-form defined on $\Omega\cup\Omega^e$ and satisfies
\begin{eqnarray}\label{Eqtransmission}
\left\{
\begin{aligned}
    &\mathbf{d}\brho + (-1)^l\star_{k^2}\bomega = \mathbf{0}&&{\rm in}\quad\Omega\cup\Omega^e,\\
    &\star_\alpha^{-1}\brho - (-1)^{(l+1)(d-1)}\mathbf{d}\bomega=\mathbf{0}&&{\rm in}\quad\Omega\cup\Omega^e,\\
    &[\star^\Gamma\mathbf{Tr}^{\mathcal{D}}(\bomega)]=[\mathbf{Tr}^{\mathcal{N}}(\bomega)] = \mathbf{0}&&{\rm on}\quad \Gamma,\\
    &\mathcal{F}(\brho,\bomega) = \mathbf{0} &&{\rm on}\quad \Gamma_R.
\end{aligned}
\right.
\end{eqnarray}
The recurrence formula for the shape derivatives of $\bomega$ under the transmission boundary condition is given
by the following theorem.

\begin{theorem}\label{TheoT}
     Let $\bomega$ be the solution of equation~\eqref{Eqtransmission} and $\delta_{\mathbf{v}_{[N]}}\bomega$ be the $N$th order shape derivative w.r.t. $\mathbf{v}_{[N]}$, with $\delta_{\mathbf{v}_{[N]}}\bomega=\bomega$ for $N=0$. Assume that the boundary conditions for $\delta_{\mathbf{v}_{[N]}}\bomega$ on $\Gamma$ are given by $ \big[\star^\Gamma\mathbf{Tr}^{\mathcal{D}}(\delta_{\mathbf{v}_{[N]}}\bomega)\big]$ and $\big[\mathbf{Tr}^\mathcal{N}(\delta_{\mathbf{v}_{[N]}}\bomega)\big]$. Then the recurrence formulas for the boundary conditions on $\Gamma$ of the $N+1$th order shape derivative $\delta_{\mathbf{v}_{[N+1]}}\bomega$ on $\Gamma$ are given by
     \begin{eqnarray}\label{recurrenceTransGen1}
    \begin{aligned}
    \big[\star^\Gamma\mathbf{Tr}^{\mathcal{D}}(\delta_{\mathbf{v}_{[N+1]}}\bomega)\big]=&\big[\mathbf{i}_{\mathbf{v}_{N+1}}\mathbf{d}\big(\star^\Gamma\mathbf{Tr}^\mathcal{D}(\delta_{\mathbf{v}_{[N]}}\bomega)\big)\\
    &+\delta_{\mathbf{v}_{N+1}}^{\bomega}\big(\star^\Gamma\mathbf{Tr}^\mathcal{D}(\delta_{\mathbf{v}_{[N]}}\bomega)\big)+\delta_{\mathbf{v}_{N+1}}^{\mathbf{n}}\big(\star^\Gamma\mathbf{Tr}^\mathcal{D}(\delta_{\mathbf{v}_{[N]}}\bomega)\big)\\
    &-\mathbf{i}_{\mathbf{v}_{N+1}}\mathbf{d}\star^\Gamma\mathbf{i}_\mathbf{n}\star\delta_{\mathbf{v}_{[N]}}\bomega - \star^\Gamma\mathbf{i}_{\delta_{\mathbf{v}_{N+1}}\mathbf{n}}\star\delta_{\mathbf{v}_{[N]}}\bomega\big],
    \end{aligned}
\end{eqnarray}
and
\begin{eqnarray}\label{recurrenceTransGen2}
 \begin{aligned}
     \big[\mathbf{Tr}^\mathcal{N}(\delta_{\mathbf{v}_{[N+1]}}\bomega)\big] =&\big[\mathbf{i}_{\mathbf{v}_{N+1}}\mathbf{d}\mathbf{Tr}^{\mathcal{N}}(\delta_{\mathbf{v}_{[N]}}\bomega)\\
     &+ \delta_{\mathbf{v}_{N+1}}^{\bomega}\mathbf{Tr}^{\mathcal{N}}(\delta_{\mathbf{v}_{[N]}}\bomega) + \delta_{\mathbf{v}_{N+1}}^{\mathbf{n}}\mathbf{Tr}^{\mathcal{N}}(\delta_{\mathbf{v}_{[N]}}\bomega)\\
    &-\mathbf{i}_{\mathbf{v}_{N+1}}\mathbf{d}\star^\Gamma_\alpha\mathbf{i}_\mathbf{n}\mathbf{d}\delta_{\mathbf{v}_{[N]}}\bomega - \star^\Gamma_\alpha\mathbf{i}_{\delta_{\mathbf{v}_{N+1}}\mathbf{n}}\mathbf{d}\delta_{\mathbf{v}_{[N]}}\bomega\big].
 \end{aligned}
 \end{eqnarray}
\end{theorem}

\begin{proof}
To find the equation satisfied by the shape derivative of $\bomega$, one needs to extend the equations~\eqref{weakform0} and~\eqref{Mdv} to the domain $\Omega\cup\Omega^e$. Let us define
\begin{eqnarray}
    \mathcal{M}_1(\brho,\bomega;\btau) = \int_{\Omega\cup\Omega^e}\star_{\alpha^{-1}}\brho\land\btau - (-1)^{(l+1)d} \int_{\Omega\cup\Omega^e}\bomega\land\mathbf{d}\btau,
\end{eqnarray}
and
\begin{eqnarray}
    \mathcal{M}_2(\bomega;\btau) = (-1)^{(d-l)}\int_{\Omega\cup\Omega^e}\star_\alpha\mathbf{d}\bomega\land\mathbf{d}\btau + (-1)^l\int_{\Omega\cup\Omega^e}\star_{k^2}\bomega\land\btau.
\end{eqnarray}
Following the derivation of equations~\eqref{weakform0} and~\eqref{Mdv}, we obtain
\begin{eqnarray}\label{transmissionw1}
\begin{aligned}
    \mathcal{M}_1(\brho,\bomega;\btau) =& (-1)^{(l+1)(d-1)}\Big(\int_{\Gamma^-}\star^\Gamma\mathbf{Tr}^{\mathcal{D}}(\bomega)\land\btau + \int_{\Gamma^-}\star\mathbf{i}_\mathbf{n}\star\mathbf{i}_\mathbf{n}\bomega\land\btau + \int_{\Gamma_R^+}\bomega\land\btau\Big)\\
    & + (-1)^{(l+1)(d-1)}\Big(\int_{\Gamma^+}\star^\Gamma\mathbf{Tr}^{\mathcal{D}}(\bomega)\land\btau + \int_{\Gamma^+}\star\mathbf{i}_\mathbf{n}\star\mathbf{i}_\mathbf{n}\bomega\land\btau\Big),
\end{aligned}
\end{eqnarray}
and
\begin{eqnarray}\label{transmissionw2}
\begin{aligned}
    \mathcal{M}_2(\bomega;\btau) =& -\int_{\Gamma^-}\star\mathbf{i}_\mathbf{n}\star\mathbf{i}_\mathbf{n}\star_\alpha\mathbf{d}\bomega\land\btau + \mathbf{Tr}^{\mathcal{N}}(\bomega)\land\btau - \int_{\Gamma_R^+}\star_\alpha\mathbf{d}\bomega\land\btau\\
    &-\int_{\Gamma^+}\star\mathbf{i}_\mathbf{n}\star\mathbf{i}_\mathbf{n}\star_\alpha\mathbf{d}\bomega\land\btau + \mathbf{Tr}^{\mathcal{N}}(\bomega)\land\btau.
\end{aligned}
\end{eqnarray}
Here, since the unit normal vectors on $\Gamma^+$ and $\Gamma^-$ are opposite, the integrands on $\Gamma^+$ and $\Gamma^-$ give the jump conditions in equation~\eqref{Eqtransmission}. By taking material derivatives of both sides of equation~\eqref{transmissionw1} w.r.t. $\mathbf{v}$ and using the same technique as in the derivation of equation~\eqref{Eq1ord}, one obtains
\begin{eqnarray}\label{transmission1ord1}
     \begin{aligned}
    \mathcal{M}_1(\delta_\mathbf{v}\brho,\delta_\mathbf{v}\bomega;\btau) =&(-1)^{(l+1)(d-1)}\Big(\int_{\Gamma^+} -\mathbf{i}_\mathbf{v}\mathbf{d}\star^\Gamma\mathbf{i}_\mathbf{n}\star\bomega\land\btau +\int_{\Gamma^+}\star\mathbf{i}_\mathbf{n}\star\mathbf{i}_\mathbf{n}\delta_\mathbf{v}\bomega\land\btau\\
    &+ \int_{\Gamma^-} -\mathbf{i}_\mathbf{v}\mathbf{d}\star^\Gamma\mathbf{i}_\mathbf{n}\star\bomega\land\btau +\int_{\Gamma^-}\star\mathbf{i}_\mathbf{n}\star\mathbf{i}_\mathbf{n}\delta_\mathbf{v}\bomega\land\btau + \int_{\Gamma_R^+}\delta_\mathbf{v}\bomega\land\btau\Big).
    \end{aligned}
\end{eqnarray}
It implies
\begin{eqnarray}\label{tranbound1}
    [\star^\Gamma\mathbf{Tr}^\mathcal{D}(\delta_\mathbf{v}\bomega)] = [-\mathbf{i}_\mathbf{v}\mathbf{d}\star^\Gamma\mathbf{i}_\mathbf{n}\star\bomega],
\end{eqnarray}
which gives the jump condition for the Dirichlet data for the first order shape derivative.

On the other hand, the jump condition for the Neumann data can be derived from equation~\eqref{transmissionw2}. By mimicking the derivation of equation~\eqref{NBD1}, one can find 
\begin{eqnarray}\label{transmission1ord2}
\begin{aligned}
    \mathcal{M}_2(\delta_\mathbf{v}\bomega;\btau) =&-\int_{\Gamma^-} -\mathbf{i}_\mathbf{v}\mathbf{d}\star^\Gamma_\alpha\mathbf{i}_\mathbf{n}\mathbf{d}\bomega\land\btau - \int_{\Gamma^-}\star\mathbf{i}_\mathbf{n}\star\mathbf{i}_\mathbf{n}\star_\alpha\mathbf{d}\delta_\mathbf{v}\bomega\land\btau - \int_{\Gamma_R^+}\star_\alpha\mathbf{d}\delta_\mathbf{v}\bomega\land\btau\\
    &-\int_{\Gamma^+} -\mathbf{i}_\mathbf{v}\mathbf{d}\star^\Gamma_\alpha\mathbf{i}_\mathbf{n}\mathbf{d}\bomega\land\btau - \int_{\Gamma^+}\star\mathbf{i}_\mathbf{n}\star\mathbf{i}_\mathbf{n}\star_\alpha\mathbf{d}\delta_\mathbf{v}\bomega\land\btau,
\end{aligned}
\end{eqnarray}
which implies
\begin{eqnarray}\label{tranbound2}
    [\mathbf{Tr}^\mathcal{N}(\delta_\mathbf{v}\bomega)] = [-\mathbf{i}_\mathbf{v}\mathbf{d}\star^\Gamma_\alpha\mathbf{i}_\mathbf{n}\mathbf{d}\bomega]. 
\end{eqnarray}

From equations~\eqref{tranbound1} and~\eqref{tranbound2}, we observe that the two jump conditions for the transmission boundary in the shape derivatives are simply the differences between the interior and exterior Dirichlet and Neumann boundary conditions on $\Gamma$. Therefore, the recurrence relations for the transmission boundary conditions, as given in Theorems~\ref{ThemD} and~\ref{ThemN}, lead to equations~\eqref{recurrenceTransGen1} and~\eqref{recurrenceTransGen2}.

\end{proof}

\section{Vector proxies of the second order shape derivatives}\label{VecPro}
    The first order shape derivative formulas have been summarized in~\cite{hiptmair2018shape} using surface differential operators. This part summarizes the boundary conditions for the second order shape derivatives in the vector proxies of Euclidean space. The corresponding proxies of $\bomega$, $\mathbf{d}$, $\mathbf{i}$ and $\land$ in $\mathbb{R}^d$, with $d=2,3$, are provided in Appendix~\ref{Appen}. Given two velocity fields $\mathbf{v}_1$ and $\mathbf{v}_2$ in $\mathbb{R}^d$, we define the three directional differential operators w.r.t. $\mathbf{v}_j$ for $j = 1,2$ as
\begin{eqnarray}\label{dirdifope}
    \mathbf{grad}_j u = \mathbf{v}_j\cdot\nabla u,\quad \mathbf{div}_j\mathbf{E}= \mathbf{v}_j\nabla\cdot\mathbf{E},\quad\mathbf{curl}_j\mathbf{E} = - \mathbf{v}_j\times\nabla\times\mathbf{E}.
\end{eqnarray}

\subsection{Vecter proxies for acoustic scattering problems}\label{vp1}
For the cases of $d = 2,3$, $l = 0$, we denote $\phi$ the incident field and $u$ the acoustic scattered field. Recall that the total field $u^t: = u + \phi$ in the impenetrable scattering problems (including Dirichlet, Neumann, and impedance boundaries) is only defined in $\Omega^e$, while in the penetrable scattering problems, $u^t$ is defined as
\begin{eqnarray}
    u^t = \left\{\begin{aligned}
        &u + \phi&&{\rm in}\quad \Omega^e,\\
        &u&&{\rm in}\quad \Omega.
    \end{aligned}\right.
\end{eqnarray}
The equations for the acoustic scattered fields are:
\begin{eqnarray}
    \begin{aligned}
        &{\rm Impenetrable\quad scattering:} && \nabla\cdot(\alpha\nabla u) + k^2 u = 0\quad{\rm in}\quad\Omega^e.\\
        &{\rm Penetrable\quad scattering:} &&\left\{\begin{aligned}
        &\nabla\cdot(\alpha^-\nabla u) + k^2 u = 0\quad{\rm in}\quad\Omega,\\
        &\nabla\cdot(\alpha^+\nabla u) + k^2 u = 0\quad{\rm in}\quad\Omega^e,
        \end{aligned}
        \right.
    \end{aligned}
\end{eqnarray}
with four different boundary conditions on $\Gamma=\partial\Omega$:
\begin{eqnarray}
    \begin{aligned}
        &{\rm Sound\ soft}: && u = -\phi.\\
        &{\rm Sound\ hard:} && \alpha(\nabla u\cdot \mathbf{n})\mathbf{n}= -\alpha(\nabla\phi\cdot \mathbf{n})\mathbf{n}.\\
        &{\rm Impedance:} && \alpha(\nabla u\cdot \mathbf{n})\mathbf{n} + i\lambda u\mathbf{n}= -\alpha(\nabla\phi\cdot \mathbf{n})\mathbf{n} - i\lambda \phi\mathbf{n}.\\
        &{\rm Transmission:} &&\left\{\begin{aligned}
        &[u] = -\phi,\\
        &[\alpha(\nabla u\cdot\mathbf{n})\mathbf{n}] = -\alpha^+(\nabla \phi\cdot \mathbf{n})\mathbf{n}.\end{aligned}\right.
    \end{aligned}
\end{eqnarray}
The boundary conditions on $\Gamma$ for the second order shape derivatives are:

\begin{tabular}{|l c l|}
    \hline
    ${\rm Sound\ soft^{(2)}}$ & &  \\ \hline
    $\delta_{\mathbf{v}_1,\mathbf{v}_2}u$ & $=$ & $-\mathbf{grad}_2\mathbf{grad}_1 u - \mathbf{grad}_2\mathbf{grad}_1\phi$\\
    & &  $ - \mathbf{grad}_1\delta_{\mathbf{v}_2}u - \mathbf{grad}_2\delta_{\mathbf{v}_1}u$ \\ \hline
    ${\rm Sound\ hard^{(2)}}$ & & \\ \hline
    $\alpha(\nabla\delta_{\mathbf{v}_1,\mathbf{v}_2}u\cdot\mathbf{n})\mathbf{n}$&$=$&$-\mathbf{div}_2\mathbf{div}_1(\alpha(\nabla u\cdot \mathbf{n})\mathbf{n}) + \alpha(\nabla u\cdot \delta_{\mathbf{v}_1,\mathbf{v}_2}\mathbf{n})\mathbf{n}$\\
    & & $-\mathbf{div}_2(\alpha(\nabla u\cdot\delta_{\mathbf{v}_1}\mathbf{n})\mathbf{n}) -\mathbf{div}_1(\alpha(\nabla u\cdot\delta_{\mathbf{v}_2}\mathbf{n})\mathbf{n})$\\
    & &$- \mathbf{div}_2(\alpha(\nabla\delta_{\mathbf{v}_1} u\cdot \mathbf{n})\mathbf{n}) + \alpha(\nabla\delta_{\mathbf{v}_1} u\cdot\delta_{\mathbf{v}_2}\mathbf{n})\mathbf{n}$\\
    & &$- \mathbf{div}_1(\alpha(\nabla\delta_{\mathbf{v}_2} u\cdot \mathbf{n})\mathbf{n}) + \alpha(\nabla\delta_{\mathbf{v}_2} u\cdot\delta_{\mathbf{v}_1}\mathbf{n})\mathbf{n}$\\
    & &$-\mathbf{div}_2\mathbf{div}_1(\alpha(\nabla \phi\cdot \mathbf{n})\mathbf{n}) + \alpha(\nabla \phi\cdot \delta_{\mathbf{v}_1,\mathbf{v}_2}\mathbf{n})\mathbf{n}$\\
    & &$-\mathbf{div}_2(\alpha(\nabla \phi\cdot\delta_{\mathbf{v}_1}\mathbf{n})\mathbf{n}) -\mathbf{div}_1(\alpha(\nabla \phi\cdot\delta_{\mathbf{v}_2}\mathbf{n})\mathbf{n})$\\ \hline
    ${\rm Impedance^{(2)}}$& & \\ \hline
    $i\lambda\delta_{\mathbf{v}_1,\mathbf{v}_2}u\mathbf{n} $ & $=$ & ${\rm\ Sound\ hard^{(2)}}-i\lambda\mathbf{di\mathbf{v}_2}\mathbf{di\mathbf{v}_1} \phi\mathbf{n}$\\
    $+\alpha(\nabla\delta_{\mathbf{v}_1,\mathbf{v}_2}u\cdot \mathbf{n})\mathbf{n}$ &  & $- i\lambda(\mathbf{di\mathbf{v}_2}\mathbf{di\mathbf{v}_1} u\mathbf{n} + \mathbf{div}_2\delta_{\mathbf{v}_1}u\mathbf{n} + \mathbf{div}_1\delta_{\mathbf{v}_2}u\mathbf{n})$ \\ \hline
    ${\rm Transmission^{(2)}}$ & &\\ \hline
    $[\delta_{\mathbf{v}_1}u]$ &$=$ & $[{\rm Sound\ soft^{(2)}}]$ \\
    $[\alpha(\nabla\delta_{\mathbf{v}_1}u\cdot\mathbf{n})\mathbf{n}]$&$=$ &$[{\rm Sound\ hard^{(2)}}]$\\ \hline
\end{tabular}

\subsection{Vector proxies for electromagnetic scattering problems}\label{vp2}
For the case of $d = 3$ and $l = 1$, we denote $\bPhi$ the incident field $\mathbf{E}$ and the scattered electric field. Similar to the acoustics cases, the total field $\mathbf{E}^t: = \mathbf{E} + \bPhi$ in the impenetrable scattering problem is defined in $\Omega^e$. In the penetrable scattering problem, $\mathbf{E}^t$ is defined as
\begin{eqnarray}
    \mathbf{E}^t = \left\{\begin{aligned}
    &\mathbf{E} + \bPhi&&{\rm in}\quad\Omega^e,\\
    &\mathbf{E}&&{\rm in}\quad\Omega.
    \end{aligned}\right.    
\end{eqnarray}
The equations for the electromagnetic scattered fields are:
\begin{eqnarray}
    \begin{aligned}
        &{\rm Impenetrable\quad scattering:} && \nabla\times(\alpha\nabla\times\mathbf{E}) - k^2\mathbf{E}  = \mathbf{0}\quad{\rm in}\quad\Omega^e.\\
        &{\rm Penetrable\quad scattering:} &&\left\{\begin{aligned}
        &\nabla\times(\alpha^-\nabla\times\mathbf{E}) - k^2\mathbf{E}  = \mathbf{0}\quad{\rm in}\quad\Omega,\\
        &\nabla\times(\alpha^+\nabla\times\mathbf{E}) - k^2\mathbf{E}  = \mathbf{0}\quad{\rm in}\quad\Omega^e,
        \end{aligned}
        \right.
    \end{aligned}
\end{eqnarray}
with four different boundary conditions on $\Gamma$:
\begin{eqnarray}
    \begin{aligned}
        &{\rm PEC:}  - \mathbf{n}\times\mathbf{n}\times\mathbf{E} &=& \mathbf{n}\times\mathbf{n}\times\bPhi.\\
        &{\rm PMC:}  - \mathbf{n}\times\mathbf{n}\times\alpha\nabla\times\mathbf{E} &=&\mathbf{n}\times\mathbf{n}\times\alpha\nabla\times\bPhi.\\
        &{\rm Impedance:}\\
        &- \mathbf{n}\times\mathbf{n}\times\alpha\nabla\times\mathbf{E} + i\lambda\mathbf{n}\times\mathbf{E} &=& \mathbf{n}\times\mathbf{n}\times\alpha\nabla\times\bPhi - i\lambda \mathbf{n}\times\bPhi.\\
        &{\rm Transmission:}\\
        &[- \mathbf{n}\times\mathbf{n}\times\mathbf{E}] &=& \mathbf{n}\times\mathbf{n}\times\bPhi.\\
        &[- \mathbf{n}\times\mathbf{n}\times\alpha\nabla\times\mathbf{E}] &=&\mathbf{n}\times\mathbf{n}\times\alpha\nabla\times\bPhi.
    \end{aligned}
\end{eqnarray}
The boundary conditions on $\Gamma$ for the second order shape derivatives are:

\begin{tabular}{|l c l|}
     \hline
     ${\rm PEC^{(2)}}$& &  \\ \hline
     $ - \mathbf{n}\times\mathbf{n}\times\delta_{\mathbf{v}_1,\mathbf{v}_2}\mathbf{E}$& $=$ &$\mathbf{curl}_2\mathbf{curl}_1\mathbf{n}\times\mathbf{n}\times\mathbf{E} + \mathbf{n}\times\delta_{\mathbf{v}_1,\mathbf{v}_2}\mathbf{n}\times\mathbf{E}$\\
        & &  $+\mathbf{curl}_2\mathbf{curl}_1\mathbf{n}\times\mathbf{n}\times\bPhi + \mathbf{n}\times\delta_{\mathbf{v}_1,\mathbf{v}_2}\mathbf{n}\times\bPhi$\\
        & &  $+ \mathbf{curl}_2\mathbf{n}\times\delta_{\mathbf{v}_1}\mathbf{n}\times\mathbf{E} + \mathbf{curl}_1\mathbf{n}\times\delta_{\mathbf{v}_2}\mathbf{n}\times\mathbf{E}$\\
        & & $+ \mathbf{curl}_2\mathbf{n}\times\delta_{\mathbf{v}_1}\mathbf{n}\times\bPhi + \mathbf{curl}_1\mathbf{n}\times\delta_{\mathbf{v}_2}\mathbf{n}\times\bPhi$\\
        & & $+ \mathbf{curl}_2\mathbf{n}\times\mathbf{n}\times\delta_{\mathbf{v}_1}\mathbf{E}+ \mathbf{n}\times\delta_{\mathbf{v}_2}\mathbf{n}\times\delta_{\mathbf{v}_1}\mathbf{E}$\\
        & & $+ \mathbf{curl}_1\mathbf{n}\times\mathbf{n}\times\delta_{\mathbf{v}_2}\mathbf{E}+ \mathbf{n}\times\delta_{\mathbf{v}_1}\mathbf{n}\times\delta_{\mathbf{v}_2}\mathbf{E}$ \\ \hline
        ${\rm PMC^{(2)}}$& &\\ \hline
        $- \mathbf{n}\times\mathbf{n}\times\alpha\nabla\times\delta_{\mathbf{v}_1,\mathbf{v}_2}\mathbf{E}$&$=$ & ${\rm Substitute}:$\\
        & & $\mathbf{E}\leftarrow\alpha\nabla\times\mathbf{E}, \quad\bPhi\leftarrow\alpha\nabla\times\bPhi,$ \\
        & &$\delta_{\mathbf{v}_1}\mathbf{E}\leftarrow\alpha\nabla\times\delta_{\mathbf{v}_1}\mathbf{E},\quad\delta_{\mathbf{v}_2}\mathbf{E}\leftarrow\alpha\nabla\times\delta_{\mathbf{v}_2}\mathbf{E}$\\
        & &${\rm in\ PEC^{(2)}}$\\ \hline
        ${\rm Impedance^{(2)}}$& &\\ \hline
        $ i\lambda\mathbf{n}\times\delta_{\mathbf{v}_1,\mathbf{v}_2}\mathbf{E}$ & $=$ & ${\rm PMC^{(2)}}$\\
        $- \mathbf{n}\times\mathbf{n}\times\alpha\nabla\times\delta_{\mathbf{v}_1,\mathbf{v}_2}\mathbf{E}$ & &$+\mathbf{curl}_2\mathbf{curl}_1(-i\lambda\mathbf{n}\times\mathbf{E}) - i\lambda \delta_{\mathbf{v}_1,\mathbf{v}_2}\mathbf{n}\times\mathbf{E}$\\
        & & $ + \mathbf{curl}_2\mathbf{curl}_1(-i\lambda\mathbf{n}\times\bPhi) - i\lambda \delta_{\mathbf{v}_1,\mathbf{v}_2}\mathbf{n}\times\bPhi$\\
        & &$-\mathbf{curl}_2(-i\lambda\delta_{\mathbf{v}_1}\mathbf{n}\times\mathbf{E}) -\mathbf{curl}_1(-i\lambda\delta_{\mathbf{v}_2}\mathbf{n}\times\mathbf{E})$\\
        & &$-\mathbf{curl}_2(-i\lambda\delta_{\mathbf{v}_1}\mathbf{n}\times\bPhi) -\mathbf{curl}_1(-i\lambda\delta_{\mathbf{v}_2}\mathbf{n}\times\bPhi)$\\
        & &$+\mathbf{curl}_2(-i\lambda\mathbf{n}\times\delta_{\mathbf{v}_1}\mathbf{E}) - i\lambda \delta_{\mathbf{v}_2}\mathbf{n}\times\delta_{\mathbf{v}_1}\mathbf{E}$\\
        & &$+\mathbf{curl}_1(-i\lambda\mathbf{n}\times\delta_{\mathbf{v}_2}\mathbf{E}) - i\lambda \delta_{\mathbf{v}_1}\mathbf{n}\times\delta_{\mathbf{v}_2}\mathbf{E}$\\ \hline
        ${\rm Transmission^{(2)}}$ & &\\ \hline
        $[- \mathbf{n}\times\mathbf{n}\times\delta_{\mathbf{v}_1,\mathbf{v}_2}\mathbf{E}]$ &$=$ & $[{\rm PEC^{(2)}}]$ \\
        $[- \mathbf{n}\times\mathbf{n}\times\alpha\nabla\times\delta_{\mathbf{v}_1,\mathbf{v}_2}\mathbf{E}]$ & $=$& $[{\rm PMC^{(2)}}]$\\ \hline
\end{tabular}

\begin{remark}
 Some of the second order formulas can be verified through the work of~\cite{hagemann2020application,harbrecht2018second}.
\end{remark}

\section{Conclusion}\label{secCon}
In this work, we derive the shape Taylor expansion for scattering problems based on exterior differential forms. Specifically, for acoustic and electromagnetic scattering problems with Dirichlet, Neumann, impedance, and transmission conditions, we present the recurrence formulas for shape derivatives up to any order. In the derivation, we introduce the trace operator decomposition to overcome the complexity induced by surface differential operators. We also provide vector proxies for the second order shape derivatives for acoustic and electromagnetic scattering problems. Results of this work can be applied to inverse scattering problems, optimal design problems, uncertainty quantification, and other problems involving shape parameters. The efficient algorithms for the shape Taylor formulas and their applications in optimizations will be explored in future work.

\appendix

\section{Vector proxies}\label{Appen}
 In this appendix, we give the vector proxies for the differential forms in $\mathbb{R}^d$, with $d = 2,3$.
\subsection{Exterior Derivative and Contraction}
\ \\
In $\mathbb{R}^2$:
\begin{eqnarray}
    \begin{aligned}
        &\bomega\in\land^0\Leftrightarrow u:\mathbb{R}^2\rightarrow\mathbb{R}, &&\mathbf{d}\bomega\Leftrightarrow\nabla u, &&\mathbf{i}_{\mathbf{v}}\bomega\Leftrightarrow0,\\
        &\bomega\in\land^1\Leftrightarrow \mathbf{u}:\mathbb{R}^2\rightarrow\mathbb{R}^2, &&\mathbf{d}\bomega\Leftrightarrow\frac{\partial u_2}{\partial x_1} - \frac{\partial u_1}{\partial x_2}, &&\mathbf{i}_{\mathbf{v}}\bomega\Leftrightarrow {\mathbf{v}}\cdot \mathbf{u},\\
        &\bomega\in\land^2\Leftrightarrow u:\mathbb{R}^2\rightarrow\mathbb{R}, &&\mathbf{d}\bomega\Leftrightarrow0, &&\mathbf{i}_{\mathbf{v}}\bomega\Leftrightarrow u[-v_2,v_1]^T.
    \end{aligned}
\end{eqnarray}
In $\mathbb{R}^3$:
\begin{eqnarray}
    \begin{aligned}
        &\bomega\in\land^0\Leftrightarrow u:\mathbb{R}^3\rightarrow\mathbb{R}, &&\mathbf{d}\bomega\Leftrightarrow\nabla u, &&\mathbf{i}_{\mathbf{v}}\bomega\Leftrightarrow0,\\
        &\bomega\in\land^1\Leftrightarrow \mathbf{u}:\mathbb{R}^3\rightarrow\mathbb{R}^3, &&\mathbf{d}\bomega\Leftrightarrow\nabla\times \mathbf{u}, &&\mathbf{i}_{\mathbf{v}}\bomega\Leftrightarrow {\mathbf{v}}\cdot \mathbf{u},\\
        &\bomega\in\land^2\Leftrightarrow \mathbf{u}:\mathbb{R}^3\rightarrow\mathbb{R}^3, &&\mathbf{d}\bomega\Leftrightarrow\nabla\cdot \mathbf{u}, &&\mathbf{i}_{\mathbf{v}}\bomega\Leftrightarrow \mathbf{u}\times {\mathbf{v}},\\
        &\bomega\in\land^3\Leftrightarrow u:\mathbb{R}^3\rightarrow\mathbb{R}, &&\mathbf{d}\bomega\Leftrightarrow0, &&\mathbf{i}_{\mathbf{v}}\bomega\Leftrightarrow u{\mathbf{v}}.
    \end{aligned}
\end{eqnarray}
\subsection{Exterior product}
\ \\
In $\mathbb{R}^2$:
\begin{eqnarray}
   \begin{aligned}
       &\bomega\in\land^0,\bbeta\in\land^{0,1,2},&&\bomega\land\bbeta\Leftrightarrow u{\mathbf{v}},\\
       &\bomega\in\land^1,\bbeta\in\land^1,&&\bomega\land\bbeta\Leftrightarrow u_1v_2 - u_2v_1.
   \end{aligned}
\end{eqnarray}
In $\mathbb{R}^3$:
\begin{eqnarray}
    \begin{aligned}
        &\bomega\in\land^0,\bbeta\in\land^{0,1,2,3},&&\bomega\land\bbeta\Leftrightarrow u{\mathbf{v}},\\
        &\bomega\in\land^1,\bbeta\in\land^1,&&\bomega\land\bbeta\Leftrightarrow \mathbf{u}\times {\mathbf{v}},\\
        &\bomega\in\land^1,\bbeta\in\land^2,&&\bomega\land\bbeta\Leftrightarrow \mathbf{u}\cdot {\mathbf{v}}.\\
    \end{aligned}
\end{eqnarray}

\subsection{Trace operators} 
\ \\
Dirichlet trace:
\begin{eqnarray}
    \begin{aligned}
        &d = 2,3,&&\bomega\in\land^0,&&\mathbf{Tr}^\mathcal{D}(\bomega) \Leftrightarrow \mathbf{n}u|_\Gamma,\\
        &d = 3,&&\bomega\in\land^1,&& \mathbf{Tr}^\mathcal{D}(\bomega) \Leftrightarrow -\mathbf{n}\times \mathbf{u}|_\Gamma.
    \end{aligned}
\end{eqnarray}
Neumann trace:
\begin{eqnarray}
\begin{aligned}
    &d = 2,3,&&\bomega\in\land^0,&&\mathbf{Tr}^\mathcal{N}(\bomega)\Leftrightarrow\mathbf{n}(\mathbf{n}\cdot\alpha\nabla u)|_\Gamma,\\
    &d = 3,&&\bomega\in\land^1,&&\mathbf{Tr}^\mathcal{N}(\bomega)\Leftrightarrow-\mathbf{n}\times\mathbf{n}\times\alpha\nabla\times \mathbf{u}|_\Gamma.
\end{aligned}
\end{eqnarray}

\bibliographystyle{plain} 
\bibliography{REF.bib} 
\end{document}